\documentclass{elsarticle}

\usepackage{lineno,hyperref}
\modulolinenumbers[5]

\usepackage[american]{babel}	
\usepackage[utf8]{inputenc}		
\usepackage[T1]{fontenc}		
\usepackage[english=american]{csquotes}	
\usepackage{graphicx}			
\usepackage{epstopdf}
\usepackage{subfig}				
\usepackage{marvosym}			
\usepackage{amsmath}			
\usepackage{amssymb}
\usepackage{amsthm}
\usepackage{mathtools}
\usepackage{color}
\usepackage[linesnumbered,ruled]{algorithm2e}
\usepackage[normalem]{ulem}
\usepackage{rotating}

\theoremstyle{definition}
\newtheorem{definition}{Definition}

\newtheorem{remark}{Remark}

\theoremstyle{plain}
\newtheorem{lemma}{Lemma}
\newtheorem{theorem}{Theorem}
\newtheorem{proposition}{Proposition}
\newtheorem{corollary}{Corollary}

\begin{document}

\begin{frontmatter}
\title{Extremal properties of the Colless balance index for rooted binary trees}

\author{Mareike Fischer\corref{cor1}}
\ead{email@mareikefischer.de}
\cortext[cor1]{Corresponding author}

\author{Lina Herbst\corref{cor2}}
\author{Kristina Wicke\corref{cor3}}

\address{Institute of Mathematics and Computer Science, University of Greifswald, Greifswald, Germany}

\begin{abstract} 

\textcolor{red}{ATTENTION: This manuscript has been subsumed by another manuscript, which can be found on Arxiv: arXiv:1907.05064}

Measures of tree balance play an important role in various research areas, for example in phylogenetics. There they are for instance used to test whether an observed phylogenetic tree differs significantly from a tree generated by the Yule model of speciation.
One of the most popular indices in this regard is the Colless index, which measures the degree of balance for rooted binary trees. 
While many statistical properties of the Colless index (e.g. asymptotic results for its mean and variance under different models of speciation) have already been discussed in different contexts, we focus on its extremal properties. While it is relatively straightforward to characterize trees with maximal Colless index, the analysis of the minimal value of the Colless index and the characterization of trees that achieve it, are much more involved. 
In this note, we therefore focus on the minimal value of the Colless index for any given number of leaves. 
We derive both a recursive formula for this minimal value, as well as an explicit expression, which shows a surprising connection between the Colless index and the so-called Blancmange curve, a fractal curve that is also known as the Takagi curve. Moreover, we characterize two classes of trees that have minimal Colless index, consisting of the set of so-called \emph{maximally balanced trees} and a class of trees that we call \emph{greedy from the bottom trees}.
Furthermore, we derive an upper bound for the number of trees with minimal Colless index by relating these trees with trees with minimal Sackin index (another well-studied index of tree balance).
\end{abstract}

\end{frontmatter}


\section{Introduction} \label{Sec_Introduction}
Rooted trees are used in different research areas, ranging from computer science (for an overview see Chapter 2.3 in \citet{Knuth1997}) to evolutionary biology. In particular, they are often used to represent the evolutionary relationships among different species, where the leaves of the tree represent a set of extant species and the root represents their most recent common ancestor. Often it is of interest to study the structure and shape of a phylogenetic tree, in particular its degree of balance. It is for example well known that the Yule model in phylogenetics (a speciation model, which assumes a pure birth process and a constant rate of speciation among species) leads to rather imbalanced trees (e.g. \citet{Mooers1997,Steel2016}). Thus, information on the balance of a tree can be used to test whether an observed phylogenetic tree is consistent with this model (see for example \citet{Mooers1997, Blum2006, Bartoszek2018}). 
Note, however, that tree balance does not only play an important role in phylogenetics, but for example also in computer science (see for example \citet{Nievergelt1973, Walker1976, Chang1984, Andersson1993, Pushpa2007}).

In order to measure the degree of balance for a tree, several balance indices have been introduced, e.g. the \emph{Sackin index} (cf. \citet{Sackin1972}), the \emph{Colless index} (cf. \citet{Colless1982}) and, more recently, the \emph{total cophenetic index} (cf. \citep{Mir2013}). While these indices differ in their definitions and computations, they all assign a single number to a tree that captures its balance.

Here, we focus on the Colless index and explore its extremal properties (for statistical properties of the Colless index see  for example \citet{Blum2006a}). 
It is relatively straightforward to analyze the maximal Colless index a tree on $n$ leaves can have. In fact, it can be shown that the so-called caterpillar tree on $n$ leaves (i.e. the unique rooted binary tree with only one \enquote{cherry}, i.e. one pair of leaves adjacent to the same node) is the unique tree with maximal Colless index (cf. Lemma 1 in \citet{Mir2018}) and this maximal value turns out to be $\frac{(n-1)(n-2)}{2}$. 

In contrast, the analysis of the minimal value of the Colless index is much more involved and is one of the main aims of this manuscript. 
On the one hand, for many years, no explicit expression for the minimal Colless index for a tree on $n$ leaves had been known. At the time when this manuscript was about to be submitted, an independent study for the first time presented such an expression (cf. \citet{Coronado2019}). However, their expression differs from the one we will derive in this manuscript. Ours is based on a fractal curve, namely the so-called Blancmange curve, and therefore shows the fractal structure and the symmetry of the minimum Colless index.\footnote{We will discuss more differences between the present manuscript and the one by \citet{Coronado2019} in the discussion section.} 
On the other hand, this minimum may be achieved by several trees, which raises the question of how many such trees there are. This question has so far not been directly addressed  anywhere in the literature. In the following, we will present an upper bound for the number of trees with minimal Colless index. Thus, our manuscript is the first study which both analyzes the minimal Colless index itself and gives an upper bound on the number of trees with minimal Colless index. Moreover, we characterize some classes of trees that achieve the minimum value. Last, in the discussion section we also show how the results of \citet{Coronado2019} can be used both to improve our bound as well as to give a recursive formula for the number of trees with minimal Colless value.

To be precise, we first prove a recursive formula for the minimal Colless index, before deriving an explicit expression. Surprisingly, the latter is strongly related to a fractal curve, the so-called \emph{Blancmange curve}. This curve is also known as the \emph{Takagi curve} (cf. \citet{Takagi1901}), which for example plays a role in number theory, combinatorics and analysis (cf. \citet{Allaart2012}).

While the recursive formula directly gives rise to a class of trees with minimal Colless index (the class of so-called \emph{maximally balanced trees} (cf. \citet{Mir2013}), we additionally introduce another class of trees with minimal Colless index, namely the class of \emph{greedy from the bottom trees}. We also show that the leaf partitioning induced by the root of these two classes of trees are extremal -- i.e. no tree with minimum Colless index can have a smaller difference between the number of leaves in the left and right subtrees than the maximally balanced tree or a larger difference between these numbers than the greedy from the bottom tree.

We then turn to the combinatorial task of determining the number of trees with minimal Colless index for a given number of leaves. 
By showing that all trees with minimal Colless index also have minimal Sackin index,
we derive an upper bound for this number, since the number of trees with minimal Sackin index is known in the literature (cf. Theorem 8 in \citet{Fischer2018}). Using some knowledge on the leaf partitioning induced by the root of trees with minimum Colless index, we are able to improve this bound even further. However, note that the connection between the Sackin index and the Colless index is intriguing also in its own right -- it shows that two of the most frequently used tree balance indices are actually closely related, even though their definition is very different and even though proving properties of the Colless index is mathematically more involved.

We end our manuscript by linking our results with the results of \citet{Coronado2019} and by pointing out some directions for future research.

\section{Basic definitions and preliminary results} \label{Sec_Preliminaries}
Before we can present our results, we need to introduce some definitions and notations. 
Throughout this manuscript a \emph{rooted tree} is a tree $T=(V(T),E(T))$ with node set $V(T)$ and edge set $E(T)$, where one node is designated as the root (called $\rho$). We use $V_L(T) \subseteq V(T)$ (with $|V_L(T)|=n$) to denote the leaf set of $T$ (i.e. $V_L(T) = \{v \in V: deg(v) \leq 1\}$) and by $\mathring{V}(T)$ we denote the set of internal nodes, i.e. $\mathring{V}(T) = V(T) \setminus V_L(T)$. If $|V_L(T)|=1$, $T$ consists of only one node and no edge and for technical reasons this node is at the same time defined to be the root and the only leaf in the tree. 
Whenever there is no ambiguity we simply denote $E(T)$, $V(T)$, $\mathring{V}(T)$ and $V_L(T)$ as $E$, $V$, $\mathring{V}$ and $V_L$.

Now, for $n \geq 2$, a \emph{rooted binary tree} is a rooted tree where the root has degree $2$ and all other internal nodes have degree $3$. For $n = 1$, we consider the unique tree consisting of only one node as rooted and binary.

Furthermore, we implicitly assume that all edges in $T$ are directed away from the root and whenever there exists a path from $u$ to $v$ in $T$, we call $u$ an \emph{ancestor} of $v$ and $v$ a \emph{descendant} of $u$. In addition, the direct descendants of a node are called \emph{children} of this node, and in a binary tree with $n \geq 2$ leaves each interior node has exactly two children. Two leaves $x$ and $y$ are said to form a cherry, denoted by $[x,y]$, if they have the same direct ancestor.

Given a node $v$ of $T$, we call the set $C_T(v)$ of its descendant leaves the \emph{cluster} of $v$ and use $\kappa_T(v)$ to denote its cardinality.
If $v$ itself is a leaf, we set $C_T(v)=\{v\}$, i.e. $\kappa_T(v)=1$.
Moreover, we denote the subtree of $T$ rooted at $v$ by $T_v$.

Recall that a rooted binary tree $T$ can be decomposed into its two maximal pending subtrees $T_a$ and $T_b$ rooted at the direct descendants $a$ and $b$ of $\rho$, and we denote this by $T=(T_a,T_b)$. We use $n_a$ and $n_b$ to denote the number of leaves of $T_a$ and $T_b$, respectively, and without loss of generality assume throughout this manuscript that $n_a \geq n_b$.

The \emph{depth} $\delta_T(v)$ of a node $v$ is the number of edges on the unique shortest path from the root to $v$. 
Additionally, the \emph{height} of a tree is defined as $h(T) = \max\limits_{v \, \in \, V_L} \delta_v$, i.e. it is the maximum depth of all leaves. 

Given a rooted binary tree $T$ and an internal node $v \in \mathring{V}$ with children $v_1$ and $v_2$, the \emph{balance value} of $v$ is defined as $bal_T(v) = |\kappa_T(v_1)-\kappa_T(v_2)|$. We call an internal node $v$ \emph{balanced} if $bal_T(v) \leq 1$. 
Based on this we call a tree on $n$ leaves \emph{maximally balanced} if all its internal nodes are balanced. Recursively, a rooted binary tree is maximally balanced if its root is balanced and both its maximal pending subtrees are maximally balanced. Note that for all $n \in \mathbb{N}$, there exists a unique maximally balanced tree on $n$ leaves (cf. \citet{Mir2013}), which we denote by $T^{mb}_n$ (cf. Figure \ref{treetop}).

Two other particular trees, which will be needed in the following, are the so-called \emph{caterpillar tree} $T^{cat}_n$ and the so-called \emph{fully balanced tree} $T^{fb}_{k}$, where the latter denotes the unique tree with $n=2^k$ leaves in which all leaves have depth precisely $k$ (cf. Figure \ref{treetop}). Note that we often call $T^{fb}_{k}$ the \emph{fully balanced tree of height $k$} and that we have $T^{fb}_{k}=(T_a, T_b)=(T^{fb}_{k-1}, T^{fb}_{k-1})$, where both $T_a$ and $T_b$ are fully balanced trees of height $k-1$.
Moreover, note that $T^{fb}_{k} = T^{mb}_{2^k}$, because in the special case of $n=2^k$, $T_k^{fb}$ is the unique tree with $bal_{T_k^{fb}}(v)=0$ for all $v \in \mathring{V}$.

The caterpillar tree on the other hand denotes the unique rooted binary tree with $n$ leaves that has only one cherry (cf. Figure \ref{treetop}).

\begin{figure}[htbp]
\centering
\includegraphics[scale=0.45]{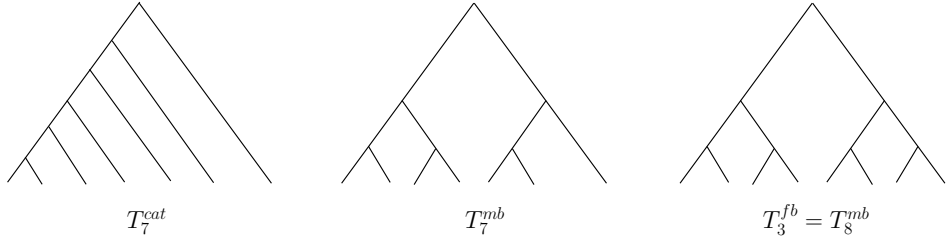}
\caption{Caterpillar tree $T^{cat}_7$ on 7 leaves, maximally balanced tree $T^{mb}_7$ on 7 leaves and fully balanced tree $T^{fb}_3 = T^{mb}_8$ on $2^3=8$ leaves.}
\label{treetop}
\end{figure}

We are now in a position to define the most important concept of this manuscript, namely the Colless index. 

\begin{definition}[\citet{Colless1982}] \label{Def_Colless}
The \emph{Colless index} of a rooted binary tree $T$ is defined as 
\begin{align*}
\mathcal{C}(T) &= \sum\limits_{v \in \mathring{V}(T)}  bal_T(v) \\
&= \sum\limits_{v \in \mathring{V}(T)} \vert  \kappa_T(v_1) - \kappa_T(v_2) \vert,
\end{align*}
where $v_1$ and $v_2$ denote the children of $v$.
\end{definition}

Note that $\mathcal{C}(T) \geq 0$ as it is defined as a sum of absolute values.
As an example consider the three trees depicted in Figure \ref{treetop}. Here, we have: $\mathcal{C}(T^{cat}_7)=15$, $\mathcal{C}(T^{mb}_7)=2$ and $\mathcal{C}(T^{fb}_3)=0$.

Note that the smaller the Colless index of a tree, the \emph{more balanced} we consider it, i.e. whenever we have for two trees $T_1$ and $T_2$ on $n$ leaves that $\mathcal{C}(T_1) < \mathcal{C}(T_2)$, then $T_1$ is called more balanced than $T_2$. 
For example, in Figure \ref{treetop}, $T_7^{mb}$ is more balanced than $T_7^{cat}$.

Also note that the Colless index of a rooted binary tree can be calculated recursively by considering the standard decomposition of a tree.

\begin{lemma} \label{colless_sum}
Let $T=(T_a,T_b)$ be a rooted binary tree. Let $n_a$ and $n_b$ denote the number of leaves of $T_a$ and $T_b$, respectively, where $n_a \geq n_b$. 
Then, we have
$$\mathcal{C}(T) = \mathcal{C}(T_a) + \mathcal{C}(T_b) +  n_a - n_b .$$
\end{lemma}
\begin{proof}
By Definition \ref{Def_Colless} we have
\begin{align*}
\mathcal{C}(T) &= \sum\limits_{v \in \mathring{V}(T)} \vert \kappa_T(v_1) - \kappa_T(v_2) \vert \\
&= \sum\limits_{v \in \mathring{V}(T_a)} \vert \kappa_{T_a}(v_1) - \kappa_{T_a}(v_2) \vert + \sum\limits_{v \in \mathring{V}(T_b)} \vert \kappa_{T_b}(v_1) - \kappa_{T_b}(v_2)\vert + \vert \kappa_T(a) - \kappa_T(b) \vert \\
&= \mathcal{C}(T_a) + \mathcal{C}(T_b) + \vert \kappa_T(a) - \kappa_T(b) \vert \\
&= \mathcal{C}(T_a) + \mathcal{C}(T_b) +  n_a - n_b .
\end{align*}
This completes the proof.
\end{proof}

This lemma has a direct consequence, which will be useful when analyzing the minimal Colless index and trees with minimal Colless index throughout this manuscript.

\begin{lemma} \label{max_subtrees}
Let $T=(T_a,T_b)$ be a rooted binary tree on $n$ leaves. Then, if $T$ has minimal Colless index for $n$, i.e. for all other trees $\widetilde{T}$ with $n$ leaves we have $\mathcal{C}(T) \leq \mathcal{C}(\widetilde{T})$, both $\mathcal{C}(T_a)$ and $\mathcal{C}(T_b)$ are minimal for $n_a$ and $n_b$, respectively. 
\end{lemma}

\begin{proof}
Assume $\mathcal{C}(T)$ is minimal. 
By Lemma \ref{colless_sum} we have
\begin{align*}
\mathcal{C}(T) &= \mathcal{C}(T_a) + \mathcal{C}(T_b) +  n_a - n_b  .
\end{align*}

Now assume that $\mathcal{C}(T_a)$ is not minimal, i.e. there is a tree $\widehat{T}$ on $n_a$ leaves such that $\mathcal{C}(\widehat{T}) < \mathcal{C}(T_a)$. 
Then we can construct a tree $\widetilde{T}$ on $n$ leaves such that $\widetilde{T} = (\widehat{T},T_b)$, i.e. we replace $T_a$ in $T$ by $\widehat{T}$ to derive $\widetilde{T}$. Now, for $\widetilde{T}$ we have by Lemma \ref{colless_sum} 
$$ \mathcal{C}(\tilde{T}) = \mathcal{C}(\hat{T}) + \mathcal{C}(T_b) +  n_a  - n_b  < \mathcal{C}(T_a) + \mathcal{C}(T_b) +  n_a - n_b  = \mathcal{C}(T).$$
This contradicts the minimality of $\mathcal{C}(T)$, which implies that the assumption was wrong.
So $\mathcal{C}(T_a)$ has to be minimal, and analogously, $\mathcal{C}(T_b)$ has to be minimal, too. 
\end{proof}

Note that Lemma \ref{max_subtrees} analogously holds for trees with maximal Colless index.

\section{Results}
The main aim of this section is to thoroughly analyze the minimal Colless index of rooted binary trees and trees that achieve it. We derive both a recursive as well as an explicit formula for the minimal Colless index and characterize two classes of trees with minimal Colless index. We end by providing an upper bound for the number of trees with minimal Colless index for any number of leaves.

\subsection{Recursive formula for the minimal Colless index}
In the following we establish a recursive formula for the minimal Colless index. Therefore, let $c_n$ denote the minimal Colless index for rooted binary trees with $n$ leaves. Then, we have the following statement:

\begin{theorem} \label{colless_minimum}
Let $c_n$ be the minimal Colless index for a rooted binary tree with $n$ leaves. Then, $c_1=c_2=0$, and for all $n \in \mathbb{N}_{\geq 1}$ we have
\begin{align*}
&c_{2n}=2c_n, \\
&c_{2n+1}=c_{n+1}+c_{n}+1.
\end{align*}
\end{theorem}

Note that the sequence of integers obtained from Theorem \ref{colless_minimum} corresponds to sequence \texttt{A296062} in the On-Line Encyclopedia of Integer Sequences (\citet{OEIS}), linking the minimal Colless index to the base-2 logarithm of the number of isomorphic maximally balanced trees with $n$ leaves (note that up to isomorphism, the maximally balanced tree is unique (\citet{Mir2013}), but there may be several isomorphic maximally balanced trees with $n$ leaves. For example, for $n=3$, there are 2 isomorphic maximally balanced trees).

\begin{proof}[Proof of Theorem \ref{colless_minimum}]
Let $n \in \mathbb{N}_{\geq 1}$. For $n=1$ there is just one tree, which consists of only one leaf and has Colless index 0. For $n=2$ again there is just one tree, which consists of one cherry and has Colless index 0 as well. Thus, $c_1=c_2=0$ as claimed.

We now show 
\begin{align*}
&(i)~ c_{2n} \leq 2c_n, \\
&(ii)~ c_{2n+1} \leq c_{n+1}+c_{n}+1.
\end{align*}
Ranging over all $n_a, n_b$ with $n_a+n_b=n$, by Lemma \ref{colless_sum} and Lemma \ref{max_subtrees} we have
\begin{align}
c_{2n}= \min\{&c_{2n-1}+c_1+2n-2, c_{2n-2}+c_2+2n-4, c_{2n-3}+c_3+2n-6, \nonumber \\
 &c_{2n-4}+c_4+2n-8, \dots, c_{n+1}+c_{n-1}+2, 2c_n \}.
\end{align}
Thus, it follows immediately that $c_{2n} \leq 2c_n$. Analogously by Lemma \ref{colless_sum} and Lemma \ref{max_subtrees}, we have
\begin{align}
c_{2n+1}= \min\{&c_{2n}+c_1+2n-1, c_{2n-1}+c_2+2n-3, c_{2n-2}+c_3+2n-5, \nonumber \\
 &c_{2n-3}+c_4+2n-7, \dots, c_{n+2}+c_{n-1}+3, c_{n+1}+c_{n}+1 \}.
\end{align}
Therefore, we derive $c_{2n+1} \leq c_{n+1}+c_{n}+1$. Thus, $(i)$ and $(ii)$ hold. \\

We now show by induction on $n$ that
\begin{align}
&c_{2n}  \geq 2c_n \text{ and } \label{ib1} \\
&c_{2n+1} \geq c_{n+1}+c_{n}+1. \label{ib2}
\end{align}

For $n=1$ we have $c_2 \geq 2c_1$, since $0 \geq 2 \cdot 0$, and $c_3 \geq c_2+c_1+1$, since $c_3 \geq 1$ (for 3 leaves there is just one tree which has Colless index 1). This completes the base case of the induction.

Now we assume that \eqref{ib1} and \eqref{ib2} hold for all natural numbers up to $n$ and show that they also hold for $n+1$. 

We start by proving \eqref{ib1}. The proof of \eqref{ib2} is given in the Appendix.

By Lemma \ref{colless_sum} and Lemma \ref{max_subtrees} we have that
\begin{align}
c_{2n+2}= \min\{&c_{2n+1}+c_1+2n, c_{2n}+c_2+2n-2, c_{2n-1}+c_3+2n-4, \dots, \nonumber \\
 &c_{n+2}+c_{n}+2, 2c_{n+1} \}. \label{i1}
\end{align}

In the following we consider two cases: $n$ even and $n$ odd. First, let $n$ be even. Thus by the inductive hypothesis, we can rewrite \eqref{i1} as
\begin{align}
c_{2n+2}\geq \min\{&c_{n+1}+c_{n}+1+c_1+2n, 2c_{n}+2c_1+2n-2, \nonumber \\
&c_{n}+c_{n-1}+1+c_2+c_1+1+2n-4, \dots, 2c_{\frac{n}{2}+1}+2c_{\frac{n}{2}}+2, 2c_{n+1} \} \nonumber \\
 = \min\{&c_{n+1}+c_{n}+c_1+n-1+n+2, 2(c_{n}+c_1+n-1), \nonumber \\
 &c_{n}+c_1+n-1+c_{n-1}+c_2+n-3+2, \dots, \nonumber \\
 &2(c_{\frac{n}{2}+1}+c_{\frac{n}{2}}+1), 2c_{n+1} \}. \label{i1.1}
\end{align}
Again by Lemma \ref{colless_sum} and Lemma \ref{max_subtrees}, we have that
\begin{align}
c_{n+1}= \min\{&c_{n}+c_1+n-1, c_{n-1}+c_2+n-3, c_{n-2}+c_3+n-5, \nonumber \\
 &\dots, c_{\frac{n}{2}+1}+c_{\frac{n}{2}}+1 \}, \label{i1.2}
\end{align}
and thus we have for example $c_{n+1} \leq c_n+c_1+n-1$. Then by using \eqref{i1.2}, \eqref{i1.1} becomes the following 
\begin{align}
c_{2n+2}\geq \min\{&c_{n+1}+c_{n+1}+\underbrace{n+2}_{\geq 0}, 2c_{n+1}, c_{n+1}+c_{n+1}+\underbrace{2}_{\geq 0}, \dots, 2c_{n+1}, 2c_{n+1} \} \nonumber \\
&= 2c_{n+1}.
\end{align}
This completes the proof of Equation \eqref{ib1} for $n$ even.

Now, let $n$ be odd. Similarly, we can rewrite \eqref{i1} by using the inductive hypothesis as
\begin{align}
c_{2n+2}\geq \min\{&c_{n+1}+c_{n}+1+c_1+2n, 2c_{n}+2c_1+2n-2, c_{n}+c_{n-1}+1+c_2+c_1+1+2n-4, \dots, \nonumber \\
 &c_{\frac{n+1}{2}+1}+c_{\frac{n+1}{2}}+1+c_{\frac{n+1}{2}}+c_{\frac{n-1}{2}}+1+2, 2c_{n+1} \} \nonumber \\
 = \min\{&c_{n+1}+c_{n}+c_1+n-1+n+2, 2(c_{n}+c_1+n-1), \nonumber \\
 &c_{n}+c_1+n-1+c_{n-1}+c_2+n-3+2, \dots, \nonumber \\
 &c_{\frac{n+1}{2}+1}+c_{\frac{n-1}{2}}+2+2c_{\frac{n+1}{2}}+2 , 2c_{n+1} \} \label{i1.3}.
\end{align}
Again by Lemma \ref{colless_sum} and Lemma \ref{max_subtrees}, we have that
\begin{align}
c_{n+1}= \min\{&c_{n}+c_1+n-1, c_{n-1}+c_2+n-3, c_{n-2}+c_3+n-5, \nonumber \\
 &\dots, 2c_{\frac{n+1}{2}} \}, \label{i1.4}
\end{align}
and thus we have for example $c_{n+1} \leq c_n+c_1+n-1$. Then by using \eqref{i1.4}, \eqref{i1.3} becomes the following
\begin{align}
c_{2n+2}\geq \min\{&c_{n+1}+c_{n+1}+\underbrace{n+2}_{\geq 0}, 2c_{n+1}, c_{n+1}+c_{n+1}+\underbrace{2}_{\geq 0}, \dots, c_{n+1}+c_{n+1}+\underbrace{2}_{\geq 0}, 2c_{n+1} \} \nonumber \\
&= 2c_{n+1}.
\end{align}
This completes the proof of Equation \eqref{ib1} for $n$ odd. 
So in both cases we obtain $c_{2n+2} \geq 2c_{n+1}$ and thus \eqref{ib1} holds for all $n$.
\\
\\
Similarly by induction, we can show that \eqref{ib2} holds for all $n$. The detailed proof is given in the Appendix.
\\ \\
Together with $(i)$ and $(ii)$ this completes the proof.
\end{proof}

\subsection{Explicit expression for the minimal Colless index}
Theorem \ref{colless_minimum} implies that we can recursively calculate the minimal Colless index for all $n$. However, we can also directly calculate it by applying the following theorem. Note that the recursions stated in Theorem \ref{colless_minimum} are needed to prove Theorem \ref{colless_explicit}.

\begin{theorem} \label{colless_explicit}
Let $c_n$ denote the minimal Colless index for a rooted binary tree with $n$ leaves. Let $k_n \coloneqq \lceil \log_2 (n) \rceil$. Then,
$$ c_n = \sum\limits_{i=0}^{k_n-2} \frac{s(2^{i-k_n+1} \cdot n)}{2^{i-k_n+1}},$$
where $s(x) = \min\limits_{z \in \mathbb{Z}}\vert  x-z \vert$, i.e. $s(x)$ is the distance from $x$ to the nearest integer.
\end{theorem}

\begin{remark}
Note that the expression above shows a surprising connection to the so-called \emph{Blancmange curve}, a fractal curve. This curve is also known as the \emph{Takagi curve} (cf. \citet{Takagi1901}) and plays a role in different areas such as combinatorics, number theory and analysis (cf. \citet{Allaart2012}) and is defined as $T: [0,1] \rightarrow \mathbb{R}$ with
$$ T(x) = \sum\limits_{i=0}^{\infty} \frac{s(2^i \cdot x)}{2^i},$$
where $s(x)$ is defined as in Theorem \ref{colless_explicit} (adapted from \citet{Allaart2012}). 
In contrast to $T: [0,1] \rightarrow \mathbb{R}$, for $c_n$ we have  $c_n: \mathbb{N} \rightarrow \mathbb{N}$, the sum of $c_n$ runs from $0$ to $k_n-2$ (and not up to infinity), and the index is shifted. So the two functions are not identical, but it is intriguing that this explicit formula for $c_n$ links the Colless index from phylogenetics to other research areas such as number theory.

Figure \ref{Fig_MinimalColless} shows the minimal Colless index $c_n$ for $n=1, \ldots, 128$ and illustrates the fractal property of it. 
\end{remark}

\begin{figure}
	\centering
	\includegraphics[scale=0.3]{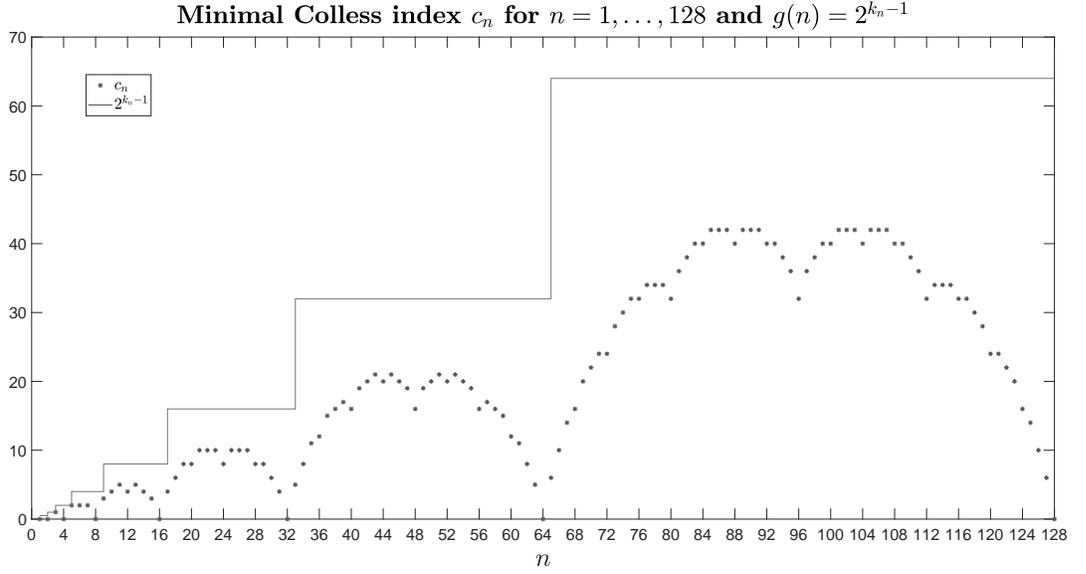}
	\caption{Plot of the minimal Colless index $c_n$ for $n=1, \ldots, 128$ and $g(n)=2^{k_n-1}$ (upper bound for the maximal minimal Colless index; cf. Lemma \ref{max_min_Colless}), where $k_n \coloneqq \lceil \log_2(n) \rceil$.}
	\label{Fig_MinimalColless}
\end{figure}

Before we give the proof of Theorem \ref{colless_explicit}, we shortly have a look at some interesting properties of $s(x) = \min\limits_{z \in \mathbb{Z}}\vert  x-z \vert$ and for example show that subadditivity holds for $s(x)$. These properties will be of relevance various times in the remainder of this manuscript.

\begin{lemma} \label{subbaditivity_of_s}
Let $s(x) = \min\limits_{z \in \mathbb{Z}}\vert  x-z \vert$, i.e. $s(x)$ is the distance from $x$ to the nearest integer.
Let $a \in \mathbb{R}$. Then, we have:
\begin{enumerate}
\item For $b \in \mathbb{R}$: $s(a+b) \leq s(a) + s(b).$
\item For $b \in \mathbb{Z}$: $s(a) = s(a+b)$.
\end{enumerate}
\end{lemma}

\begin{proof}
\begin{enumerate}
\item Let $a, b \in \mathbb{R}$. Then,
\begin{align*}
s(a) + s(b) &= \min\limits_{z_a \in \mathbb{Z}} |a-z_a| + \min\limits_{z_b \in \mathbb{Z}} |b-z_b| \\
&= \min\limits_{z_a, z_b \in \mathbb{Z}} \left\lbrace |a-z_a| + |b-z_b| \right\rbrace \\
&\geq \min\limits_{z_a, z_b \in \mathbb{Z}}  |a-z_a+b-z_b|\, \text{ (subadditivity of the absolute value)} \\
&= \min\limits_{z_a, z_b \in \mathbb{Z}} |a+b-(z_a+z_b)| \\
&= \min\limits_{z_a + z_B \in \mathbb{Z}} |a+b-(z_a+z_b)| \\
&= s(a+b).
\end{align*}
\item Now let $a \in \mathbb{R}$ and $b \in \mathbb{Z}$. By 1. we immediately have $s(a+b) \leq s(a) + s(b)$. The fact that $b \in \mathbb{Z}$ leads to $s(b)=0$, which results in $s(a+b) \leq s(a)$.
Additionally, again by 1., we have for $b \in \mathbb{Z}$ that $s(a) = s(a+b-b) \leq s(a+b) + s(-b) = s(a+b)$. Combining both arguments results in $s(a) = s(a+b)$ for all $b \in \mathbb{Z}$.
\end{enumerate}
\end{proof}

The proof of Theorem \ref{colless_explicit} requires one more lemma, the proof of which is given in the Appendix.
\begin{lemma} \label{lemma_fi}
Let $n \in (2^{k_n-1},2^{k_n})$ be odd, where $k_n \coloneqq \lceil \log_2 (n) \rceil$, and let $n-1>2^{k_n-1}$. Moreover, let $s(x) = \min\limits_{z \in \mathbb{Z}} \vert x-z \vert$ and let 
$$ f_i(n) \coloneqq \frac{s(2^{i-k_n+1} \cdot n)}{2^{i-k_n+1}}.$$
Then for $0 \leq i \leq k_n-3$,
$$f_i(n+1) + f_i(n-1)=2 \cdot f_i(n).$$
\end{lemma}

We are now in the position to prove Theorem \ref{colless_explicit}.
\begin{proof}[Proof of Theorem \ref{colless_explicit}]
Let $c_n$ denote the minimal Colless index for a rooted binary tree $T_n$ with $n$ leaves and let $k_n \coloneqq \lceil \log_2 (n) \rceil$. 
Moreover, let 
$$f_i(n) \coloneqq \frac{s(2^{i-k_n+1} \cdot n)}{2^{i-k_n+1}}.$$
Then, the statement in Theorem \ref{colless_explicit} becomes 
$c_n = \sum\limits_{i=0}^{k_n-2} f_i(n)$.
We prove this statement by induction on $n$. If $n=1$, there is only one rooted binary tree $T_1$ consisting of only one leaf. Thus, $\mathcal{C}(T_1)=0$, which is minimal (since there is only one tree for $n=1$) and thus $\mathcal{C}(T_1)=c_1$. On the other hand, we have $k_1= \lceil \log_2 (1) \rceil=0$ and thus $\sum\limits_{i=0}^{k_n-2} f_i(n)= \sum\limits_{i=0}^{-2} f_i(n)$ is the empty sum, which  by convention is $0$. This completes the base case of the induction. 

Now, suppose that the claim holds for all rooted binary trees with  up to $n-1$ leaves and consider a rooted binary tree $T_n$ with $n$ leaves. We now distinguish two cases: $n$ even and $n$ odd.

If $n$ is even, we have:
\begin{align*}
c_n &= 2 \cdot c_{\frac{n}{2}} \; \text{ by Theorem } \ref{colless_minimum} \\
&= 2 \cdot \sum\limits_{i=0}^{k_n-3} f_i \left( \frac{n}{2} \right) \\
&\text{by the inductive hypothesis and the fact that }\\
& k_{\frac{n}{2}}=  \left\lceil \log_2 \left( \frac{n}{2} \right) \right\rceil= \lceil \log_2 (n) - \log_2 (2) \rceil = \lceil \log_2 (n) \rceil - 1 = k_n-1 \\
&= 2 \cdot \sum\limits_{i=0}^{k_n-3} \frac{s(2^{i-(k_n-1)+1} \cdot \frac{n}{2})}{2^{i-(k_n-1)+1}} \\
&= \sum\limits_{i=0}^{k_n-3} \frac{s(2^{i-k_n+1} \cdot n)}{2^{i-k_n+1}} \\
&= \sum\limits_{i=0}^{k_n-3} \frac{s(2^{i-k_n+1} \cdot n)}{2^{i-k_n+1}} + \frac{s(2^{k_n-2-k_n+1} \cdot n)}{2^{k_n-2-k_n+1}} - \frac{s(2^{k_n-2-k_n+1} \cdot n)}{2^{k_n-2-k_n+1}} \\
&= \sum\limits_{i=0}^{k_n-2} \frac{s(2^{i-k_n+1} \cdot n)}{2^{i-k_n+1}} - \frac{s(\frac{n}{2})}{2^{-1}} \\
&= \sum\limits_{i=0}^{k_n-2} \frac{s(2^{i-k_n+1} \cdot n)}{2^{i-k_n+1}} \\
&\text{because } s \left( \frac{n}{2} \right) = 0 \text{ by Definition of $s(x)$ (as $n$ is even, $\frac{n}{2} \in \mathbb{Z}$)}.
\end{align*} 
This completes the proof for $n$ even. \\

If $n$ is odd, we already have by Theorem \ref{colless_minimum} that $c_n = c_{\frac{n+1}{2}} + c_{\frac{n-1}{2}} +1$. Moreover, as $k_n = \lceil \log_2 (n) \rceil$ and by the fact that $n$ is odd, we have $n \in (2^{k_n-1},2^{k_n})$, and therefore $n+1 \in (2^{k_n-1},2^{k_n}]$, which gives us $\lceil \log_2 (n+1) \rceil = k_n$. This leads to $k_{\frac{n+1}{2}} = \lceil \log_2 (\frac{n+1}{2}) \rceil= \lceil \log_2 (n+1) - \log_2 (2) \rceil = k_n-1$.

Again by the fact that $n$ is odd, we have $n-1 \in [2^{k_n-1},2^{k_n})$. \\

So first, if $n-1 = 2^{k_n-1}$ we have $\lceil \log_2 (n-1) \rceil = k_n-1$. 
Moreover, we have $k_{\frac{n-1}{2}} = \lceil \log_2 (\frac{n-1}{2}) \rceil = \lceil \log_2 (n-1) - \log_2 (2) \rceil = k_n -1 -1 = k_n -2$.
As $n-1 = 2^{k_n-1}$, we have $n= 2^{k_n-1} +1$.
This leads to
\begin{align*}
c_n &= c_{\frac{n+1}{2}} + c_{\frac{n-1}{2}} +1 \; \text{ by Theorem } \ref{colless_minimum} \\
&= \sum\limits_{i=0}^{k_{\frac{n+1}{2}}-2} f_i \left( \frac{n+1}{2} \right) + \sum\limits_{i=0}^{k_{\frac{n-1}{2}}-2} f_i \left( \frac{n-1}{2}  \right) + 1
\; \text{ by the inductive hypothesis} \\
&= \sum\limits_{i=0}^{k_n-3} \frac{s(2^{i-(k_n-1)+1} \cdot \frac{n+1}{2})}{2^{i-(k_n-1)+1}} + \sum\limits_{i=0}^{k_n-4} \frac{s(2^{i-(k_n-2)+1} \cdot \frac{n-1}{2})}{2^{i-(k_n-2)+1}} + 1 \\
&\text{by } k_{\frac{n+1}{2}} = k_n-1 \text{ and } k_{\frac{n-1}{2}} = k_n-2 \\
&= \sum\limits_{i=0}^{k_n-3} \frac{s(2^{i-k_n+1} \cdot (n+1))}{2^{i-k_n+2}} + \sum\limits_{i=0}^{k_n-4} \frac{s(2^{i-k_n+2} \cdot (n-1))}{2^{i-k_n+3}} + 1 \\
&= \sum\limits_{i=0}^{k_n-3} \frac{s(2^{i-k_n+1} \cdot (2^{k_n-1}+2))}{2^{i-k_n+2}} + \sum\limits_{i=0}^{k_n-4} \frac{s(2^{i-k_n+2} \cdot 2^{k_n-1})}{2^{i-k_n+3}} + 1 \; \text{ as } n= 2^{k_n-1} +1 \\
&= \sum\limits_{i=0}^{k_n-3} \frac{s(2^i + 2^{i-k_n+2})}{2^{i-k_n+2}} + \sum\limits_{i=0}^{k_n-4} \frac{s(2^{i+1})}{2^{i-k_n+3}} + 1 \\
&= \sum\limits_{i=0}^{k_n-3} \frac{s(2^i + 2^{i-k_n+2})}{2^{i-k_n+2}} + 1 \, \text{ as } s(2^{i+1})= 0 \text{ for all } i \geq 0 \\
&= \sum\limits_{i=0}^{k_n-3} \frac{s(2^{i-k_n+2})}{2^{i-k_n+2}} + 1 
\text{ by Lemma \ref{subbaditivity_of_s}, Part 2, as } 2^i \in \mathbb{Z}.
\end{align*}
Note that $2^{i-k_n+2} \leq 2^{(k_n-3)-k_n+2} = \frac{1}{2}$ results in $s(2^{i-k_n+2}) = 2^{i-k_n+2} - 0 = 2^{i-k_n+2}$. Therefore,
\begin{align*}
c_n &= \sum\limits_{i=0}^{k_n-3} \frac{s(2^{i-k_n+2})}{2^{i-k_n+2}} + 1 = \sum\limits_{i=0}^{k_n-3} \frac{2^{i-k_n+2}}{2^{i-k_n+2}} + 1 \\
&= \sum\limits_{i=0}^{k_n-3} 1 + 1 = (k_n -2) +1 = k_n -1.
\end{align*}
This gives us 
\begin{equation} \label{cn_nodd1}
c_n = k_n -1.
\end{equation}

We now show that this indeed equals
\begin{align*}
\sum\limits_{i=0}^{k_n-2} \frac{s(2^{i-k_n+1} \cdot n)}{2^{i-k_n+1}} 
&= \sum\limits_{i=0}^{k_n-2} \frac{s(2^{i-k_n+1} \cdot (2^{k_n-1} +1))}{2^{i-k_n+1}} \; \text{ as } n= 2^{k_n-1} +1 \\
&= \sum\limits_{i=0}^{k_n-2} \frac{s(2^i + 2^{i-k_n+1})}{2^{i-k_n+1}} \\
&= \sum\limits_{i=0}^{k_n-2} \frac{s(2^{i-k_n+1})}{2^{i-k_n+1}} \; \text{ by Lemma \ref{subbaditivity_of_s}, Part 2, as } 2^i \in \mathbb{Z}.
\end{align*}
Similarly, we have $2^{i-k_n+1} \leq 2^{(k_n-2)-k_n+1} = \frac{1}{2}$ and thus $s(2^{i-k_n+1}) = 2^{i-k_n+1}$, which leads to
\begin{align*}
\sum\limits_{i=0}^{k_n-2} \frac{s(2^{i-k_n+1})}{2^{i-k_n+1}} &= \sum\limits_{i=0}^{k_n-2} \frac{2^{i-k_n+1}}{2^{i-k_n+1}} = \sum\limits_{i=0}^{k_n-2} 1 = k_n -1 = c_n \; \text{ by \eqref{cn_nodd1}.}
\end{align*}
This completes the proof for $n-1 = 2^{k_n-1}$. \\

The last case we need to consider is the case that $n$ is odd and $n-1 \in  (2^{k_n-1},2^{k_n})$. In this case, we have $\lceil \log_2 (n-1) \rceil = k_n$ (since $n-1 > 2^{k_n-1}$). Furthermore, $k_{\frac{n-1}{2}} = \lceil \log_2 (\frac{n-1}{2}) \rceil = \lceil \log_2 (n-1) - \log_2 (2) \rceil = k_n -1$.

Therefore,
\begin{align*}
c_n &= c_{\frac{n+1}{2}} + c_{\frac{n-1}{2}} +1 
\; \text{ by Theorem } \ref{colless_minimum} \\
&= \sum\limits_{i=0}^{k_{\frac{n+1}{2}}-2} f_i \left( \frac{n+1}{2} \right) + \sum\limits_{i=0}^{k_{\frac{n-1}{2}}-2} f_i \left( \frac{n-1}{2}\right) + 1 \; \text{ by the inductive hypothesis} \\
&= \sum\limits_{i=0}^{k_n-3} \frac{s(2^{i-(k_n-1)+1} \cdot \frac{n+1}{2})}{2^{i-(k_n-1)+1}} + \sum\limits_{i=0}^{k_n-3} \frac{s(2^{i-(k_n-1)+1} \cdot \frac{n-1}{2})}{2^{i-(k_n-1)+1}} + 1 \\
&\text{by } k_{\frac{n+1}{2}} = k_{\frac{n-1}{2}} = k_n-1 \\
&= \sum\limits_{i=0}^{k_n-3} \frac{s(2^{i-k_n+1} \cdot (n+1))}{2^{i-k_n+2}} + \sum\limits_{i=0}^{k_n-3} \frac{s(2^{i-k_n+1} \cdot (n-1))}{2^{i-k_n+2}} + 1 \\
&= \frac{1}{2} \cdot \sum\limits_{i=0}^{k_n-3} \frac{s(2^{i-k_n+1} \cdot (n+1))}{2^{i-k_n+1}} + \frac{1}{2} \cdot \sum\limits_{i=0}^{k_n-3} \frac{s(2^{i-k_n+1} \cdot (n-1))}{2^{i-k_n+1}} + 1 \\
&= \frac{1}{2} \cdot \sum\limits_{i=0}^{k_n-3} f_i(n+1) + \frac{1}{2} \cdot \sum\limits_{i=0}^{k_n-3} f_i(n-1) + 1 \\
&= \frac{1}{2} \cdot \sum\limits_{i=0}^{k_n-3} (f_i(n+1) + f_i(n-1)) + 1 \\
&= \frac{1}{2} \cdot \sum\limits_{i=0}^{k_n-3} 2 \cdot f_i(n) + 1 \; \text{ by Lemma } \ref{lemma_fi} \\
&= \sum\limits_{i=0}^{k_n-3} f_i(n) + 1 \\
&= \sum\limits_{i=0}^{k_n-3} f_i(n) + 1 + \left( f_{k_n-2}(n) - f_{k_n-2}(n) \right) \\
&= \sum\limits_{i=0}^{k_n-2} f_i(n) + 1 - f_{k_n-2}(n) \\
&= \sum\limits_{i=0}^{k_n-2} f_i(n) \\
&\text{because } f_{k_n-2}(n)= \frac{s(2^{k_n-2-k_n+1} \cdot n)}{2^{k_n-2-k_n+1}} = \frac{s(2^{-1} \cdot n)}{2^{-1}} = 1 \text{ for } n \text{ odd (as $s(2^{-1}\cdot n) = \frac{1}{2}$).}
\end{align*}

So the claim holds for all cases, which completes the proof. 
\end{proof}

The following proposition states some properties of the minimal Colless index $c_n$. We use the explicit expression for the minimal Colless index stated in Theorem \ref{colless_explicit} to show the third part of this proposition. 

\begin{proposition} \label{min_colless_properties}
Let $n \in \mathbb{N}_{\geq 1}$ and let $k_n \coloneqq \lceil \log_2(n) \rceil$. Then, we have for the minimal Colless index $c_n$:
\begin{enumerate}
\item If $n=2^{k_n}+1$, then $c_n=k_n$.
\item If $n=2^{k_n}-1$, then $c_n=k_n-1$.
\item For $n \in (2^{k_n-1},2^{k_n})$ and $j \in \{1,\dots, 2^{k_n-1}-1 \}$ we have $c_{2^{k_n-1}+j} = c_{2^{k_n}-j}.$
\end{enumerate}
\end{proposition}
The proof of Proposition \ref{min_colless_properties} is given in the Appendix.
Note, however, that these properties are also reflected in Figure \ref{Fig_MinimalColless}.

\subsection{Classes of trees with minimal Colless index} \label{sec_classes}
Now that we have analyzed the minimal Colless index $c_n$, we turn our attention to trees that achieve it. 
Before considering the class of maximally balanced trees and  introducing the class of greedy from the bottom trees for arbitrary $n$, we start with analyzing the special case where $n$ is a power of two. In particular, we consider the fully balanced tree $T_k^{fb}$ of height $k$ and show that its name is indeed justified in the sense that it is the unique tree with minimal Colless index for $n=2^k$ leaves. This observation has been stated in the literature several times without formal proof (e.g. \citet{Heard1992, Rogers1993, Mir2013, Mir2018}), which is why we provide a formal proof in the following.

\begin{theorem} \label{min_colless}
Let $T$ be a rooted binary tree with $n=2^k$ leaves. Then, we have: \\
$ \mathcal{C}(T) = 0 $ if and only if $ T = T_k^{fb} $.
\end{theorem}

\begin{proof}
First, suppose that $T=T^{fb}_k$.  Then by definition of $T^{fb}_k$, we have $bal_{T_k^{fb}}(v)=0$ for all $v \in \mathring{V}$. Thus, $\mathcal{C}(T_k^{fb}) = \sum\limits_{v \in \mathring{V}} bal_{T_k^{fb}}(v) = 0$.

Now, let $T$ be a rooted binary tree with $n=2^k$ leaves and suppose that $\mathcal{C}(T)=0$. 
We prove by induction on $k$ that $\mathcal{C}(T)=0$ implies that $T$ equals $T_k^{fb}$. For $k=0$, we have $2^0=1=n$, and there is only one rooted binary tree $T$, which is by definition a fully balanced tree. 
We now assume that the statement holds up to $k-1$ and consider a rooted binary tree $T=(T_a, T_b)$ with $n=2^k$ leaves. Note that without loss of generality, $n \geq 2$ (else we consider the base case of the induction again). 

By Lemma \ref{colless_sum} we have
\begin{align*}
0 &= \mathcal{C}(T_a) + \mathcal{C}(T_b) +  n_a - n_b,
\end{align*}
which implies $\mathcal{C}(T_a) = \mathcal{C}(T_b) = (n_a-n_b)=0$. In particular, $n_a=n_b=2^{k-1}$.  
Therefore, by the inductive hypothesis, we have that $T_a$ and $T_b$ are both fully balanced trees of height $k-1$. This concludes that $T$ is the fully balanced tree of height $k$.
This completes the proof.
\end{proof}

\subsubsection{Maximally balanced trees} \label{sec_mb}
While we have already seen that for $n=2^k$ there is exactly one tree with minimal Colless index, for arbitrary $n$ there might be several ones.
 
In Theorem \ref{colless_minimum} we have already seen how to calculate the minimal Colless index for a rooted binary tree with $n$ leaves recursively. This theorem directly yields a construction principle for trees with minimal Colless index. 
Given a number $n$ of leaves, we construct a tree $T = (T_a, T_b)$ with $\mathcal{C}(T)=c_n$, where
\begin{itemize}
\item $n_a = n_b = \frac{n}{2}$, if $n$ is even;
\item $n_a = \frac{n+1}{2}$ and $n_b = \frac{n-1}{2}$, if $n$ is odd,
\end{itemize}
and where $T_a$ and $T_b$ are constructed recursively by the same principle. 
In particular, this implies that for every internal node $v$ of $T$ we have: $bal_T(v) \leq 1$, which in turn implies that $T$ is the maximally balanced tree.

Note that this approach can be seen as a \enquote{greedy from the top strategy}\footnote{In Section \ref{sec_gfb} we will additionally consider a \enquote{greedy from the bottom strategy} for building trees with minimal Colless index.} of bipartitioning the leaf set of each subtree (starting at the root and going towards the leaves) into two sets such that the difference of their cardinalities is minimized.

By using Theorem \ref{colless_minimum}, we will now formally show that trees constructed according to this principle, i.e. maximally balanced trees, indeed have minimal Colless index.

\begin{theorem} \label{colless_minimum_mb}
Let $T_n^{mb}$ be the maximally balanced tree with $n$ leaves. Then, $T_n^{mb}$ has minimal Colless index, that is $\mathcal{C}(T_n^{mb})=c_n$. 
\end{theorem}
\begin{proof}
The proof is by induction on $n$. If $n=1$, $T_1^{mb}$ consists only of a leaf and thus $\mathcal{C}(T_1^{mb})=0=c_1$, which completes the base case of the induction. \\
Now we assume that for all maximally balanced trees with up to $n$ leaves the claim holds and consider the maximally balanced tree $T_{n+1}^{mb}$ with $n+1$ leaves. By definition all internal nodes of $T_{n+1}^{mb}$ are balanced. Thus, it remains to show that $\mathcal{C}(T_{n+1}^{mb})=c_{n+1}$. \\
If $n+1$ is even, then $n_a=n_b=\frac{n+1}{2}$ and we have
\begin{align*}
\mathcal{C} \left( T_{n+1}^{mb} \right) &= \mathcal{C} \left( T_{\frac{n+1}{2}}^{mb} \right) + \mathcal{C} \left( T_{\frac{n+1}{2}}^{mb}  \right) + \frac{n+1}{2} - \frac{n+1}{2} && \text{ by Lemma \ref{colless_sum}} \\
&=2~\mathcal{C} \left( T_{\frac{n+1}{2}}^{mb} \right) \\
&=2~ c_{\frac{n+1}{2}} &&\text{by the inductive hypothesis} \\
&=c_{n+1} &&\text{by Theorem } \ref{colless_minimum}.
\end{align*}
If $n+1$ is odd, then $n_a = \frac{n+2}{2}$ and $n_b = \frac{n}{2}$. Thus,
\begin{align*}
\mathcal{C} \left( T_{n+1}^{mb} \right) &= \mathcal{C} \left( T_{\frac{n+2}{2}}^{mb} \right) + \mathcal{C} \left( T_{\frac{n}{2}}^{mb} \right) + \frac{n+2}{2} - \frac{n}{2}  && \text{by Lemma } \ref{colless_sum} \\
&=\mathcal{C} \left( T_{\frac{n+2}{2}}^{mb} \right)+\mathcal{C} \left( T_{\frac{n}{2}}^{mb} \right)+1 \\
&=c_{\frac{n+2}{2}} + c_{\frac{n}{2}} +1 && \text{by the inductive hypothesis} \\
&=c_{n+1} &&\text{by Theorem } \ref{colless_minimum}.
\end{align*}
Therefore, in both cases $T_{n+1}^{mb}$ has minimal Colless index, which completes the proof.
\end{proof}

Note, however, that the maximally balanced tree on $n$ leaves is not necessarily the only tree with minimal Colless index. For example, both trees depicted in Figure \ref{Fig_ExampleGFB} have minimal Colless index; one of them is the maximally balanced tree on $6$ leaves, the other one is a so-called greedy from the bottom tree, which we will introduce in the following section.

\subsubsection{Greedy from the bottom trees} \label{sec_gfb}
We now introduce another class of trees with minimal Colless index, which we call \emph{greedy from the bottom} trees, or \emph{GFB trees} for short. These trees can be constructed according to the following algorithm:

\begin{algorithm}[H] \label{alg_gfb}
\SetAlgoLined
$n \leftarrow$ number of taxa\;
$treeset \leftarrow n$  trees consisting of one node each\;
$min \leftarrow 1$ \tcp*[f]{minimal tree size (number of leaves) in set of trees}\;
\While{ $\vert treeset \vert >1$}{
	$u \leftarrow$ tree from $treeset$ of size $min$\;
	$treeset = treeset \setminus \{u\}$\;
	$min \leftarrow$ minimal size of all trees in $treeset$\;
	$v \leftarrow$ tree from $treeset$ of size $min$\;
	$treeset = treeset \setminus \{v\}$\;
	$newtree \leftarrow$ tree consisting of new root $\rho_{uv}$ and maximal pending subtrees $u$ and $v$\;
	$treeset \leftarrow treeset \cup \{newtree\}$\;
	$min \leftarrow$ minimal size of all trees in $treeset$\;
}
$finaltree \leftarrow treeset[1]$\tcp*[f]{i.e. only remaining element of treeset}\;
\Return $finaltree$\;
\caption{Greedy from the bottom}
\label{gfb}
\end{algorithm}

\begin{definition}[GFB tree]
Let $T$ be a rooted binary tree on $n$ leaves that results from Algorithm \ref{gfb}. Then, $T$ is called \emph{greedy from the bottom tree} or \emph{GFB tree} for short and is denoted by $T^{gfb}_n$.
\end{definition}

Note that Algorithm \ref{alg_gfb} greedily clusters trees of minimal size starting with single nodes and proceeding until only one tree is left. Thus, in principle it goes from the leaves towards the root (in contrast to the greedy strategy presented in Section \ref{sec_mb}, which goes from the root towards the leaves), which is why we call the resulting trees GFB trees.

\begin{remark}
Note that for $n=1$ Algorithm \ref{alg_gfb} returns one single node.
Moreover, note that the GFB tree is unique for all positive integers $n$ (as it results from Algorithm \ref{alg_gfb}, which is deterministic). 
\end{remark}

The following lemma shows that both maximal pending subtrees of a GFB tree are also GFB trees, which is an intuitive property and follows directly from Algorithm \ref{alg_gfb}.

\begin{lemma} \label{gfb_subtrees}
Let $T$ be a GFB tree with $n \geq 2$ leaves and standard decomposition $T=(T_a,T_b)$. Then, $T_a$ and $T_b$ are also GFB trees.
\end{lemma}

\begin{proof}
Let $T = (T_a,T_b)$ be a GFB tree and let $n_a$ and $n_b$ denote the number of leaves of $T_a$ and $T_b$, respectively.
This means that Algorithm \ref{alg_gfb} induces a bipartition of the $n$ leaves into two disjoint sets of sizes $n_a$ and $n_b$, respectively. Note that up to the last iteration of the while-loop in Algorithm \ref{alg_gfb} these two sets are independent of each other. 

Now, applying Algorithm \ref{alg_gfb} to $n_a$ and $n_b$ leaves, respectively, results in two unique GFB trees $T^{gfb}_{n_a}$ and $T^{gfb}_{n_b}$ with $n_a$ and $n_b$ leaves.  
As the leaf sets of $T_a$ and $T_b$ of sizes $n_a$ and $n_b$, respectively, are independent and do not influence each other, this implies $T_a=T^{gfb}_{n_a}$ and $T_b=T^{gfb}_{n_b}$. This completes the proof.
\end{proof}

\begin{figure}[htbp]
	\centering
	\includegraphics[scale=0.175]{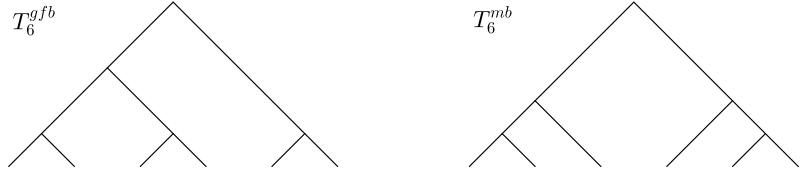}
	\caption{\emph{GFB tree} and \emph{maximally balanced tree} on 6 leaves. Both trees have minimal Colless index, namely $c_6 = \mathcal{C}(T^{gfb}_6)=\mathcal{C}(T^{mb}_6)=2$.}
	\label{Fig_ExampleGFB}
\end{figure}

In the following we will see that GFB trees always have minimal Colless index, even though in general the GFB tree $T_n^{gfb}$ is different from the maximally balanced tree $T_n^{mb}$ on $n$ leaves (cf. Figure \ref{Fig_ExampleGFB}). 
We can, however, characterize GFB trees in terms of the sizes of their maximal pending subtrees. This information will be very useful in subsequent analyses, in particular when showing that GFB trees have minimal Colless index (cf. Theorem \ref{GFB_is_minimal}).

\begin{proposition} \label{GFB_Decomposition}
Let $T_n^{gfb}$ be a GFB tree with $n\geq 2$ leaves and standard decomposition $T_n^{gfb}=(T_a,T_b)$. Let $n_a$ and $n_b$ denote the number of leaves of $T_a$ and $T_b$, respectively, such that without loss of generality, $n_a \geq n_b$. Let $k_n \coloneqq \lceil \log_2 (n) \rceil$, i.e. $n \in (2^{k_n-1}, 2^{k_n}]$. Then, we have:
	\begin{enumerate}
	\item If $n \in (2^{k_n-1}, 3 \cdot 2^{k_n-2})$, we have $n_a = n-2^{k_n-2}$ and $n_b = 2^{k_n-2}$. In particular, $T_b$ is the fully balanced tree of height $k_n-2$ and we have $k_{n_a} \coloneqq \lceil \log_2 (n_a) \rceil=k_n-1$. 
	\item If $n = 3 \cdot 2^{k_n-2}$, we have $n_a = 2^{k_n-1}$ and $n_b = 2^{k_n-2}$. In particular, $T_a$ is the fully balanced tree of height $k_n-1$ and $T_b$ is the fully balanced tree of height $k_n-2$.
	\item If $n \in (3 \cdot 2^{k_n-2}, 2^{k_n}]$, we have $n_a = 2^{k_n-1}$ and $n_b=n-2^{k_n-1}$. In particular, $T_a$ is the fully balanced tree of height $k_n-1$ and we have $k_{n_b} \coloneqq \lceil \log_2 (n_b) \rceil= k_n-1$.
	\end{enumerate}
\end{proposition}

The proof of Proposition \ref{GFB_Decomposition} is given in the appendix. However, it has an interesting consequence: 

\begin{corollary}  \label{cor_gfb}
Let $n \in \mathbb{N}_{\geq 2}$ and $n \neq 3 \cdot 2^{k_n-2}$, where $k_n = \lceil \log_2(n) \rceil$. 
Then, $T_{n-1}^{gfb}, T_{n}^{gfb}$ and $T_{n+1}^{gfb}$ have a common maximal pending subtree, which is fully balanced: 
	\begin{enumerate}[\rm (i)]
	\item If $n \in (2^{k_n-1}, 3 \cdot 2^{k_n-2})$, this common subtree is a fully balanced tree of height $k_n-2$.
	\item If $n \in (3 \cdot 2^{k_n-2}, 2^{k_n}]$, this common subtree is a fully balanced tree of height $k_n-1$.
	\end{enumerate}
\end{corollary}

\begin{proof}
The claimed statements are a direct consequence of Proposition  \ref{GFB_Decomposition}:
\begin{enumerate}[(i)]
\item If $n \in (2^{k_n-1}, 3 \cdot 2^{k_n-2})$, we distinguish between two cases:
	\begin{itemize}
	\item If $n-1 > 2^{k_n-1}$, i.e. if $n-1$, $n$ and $n+1$ are all in $(2^{k_n-1}, 3 \cdot 2^{k_n-2}]$, by Proposition \ref{GFB_Decomposition}, Parts 1 and 2,  $T_{n-1}^{gfb}, T_n^{gfb}$ and $T_{n+1}^{gfb}$ all contain a fully balanced tree of height $k_n-2$ as a maximal pending subtree.
	\item If $n-1 = 2^{k_n-1}$, by Proposition \ref{GFB_Decomposition}, Part 3, $T_{n-1}^{gfb}$ contains a fully balanced tree of height $(k_n-1)-1 = k_n-2$ as a maximal pending subtree (since $k_{n-1} = \lceil \log_2(n-1) \rceil = \lceil \log_2(2^{k_n-1}) \rceil = k_n-1$). Moreover, by Proposition \ref{GFB_Decomposition}, Part 1, $T_n^{gfb}$ and $T_{n+1}^{gfb}$ also contain a fully balanced tree of height $k_n-2$ as a maximal pending subtree.
 	\end{itemize}

\item If $n \in (3 \cdot 2^{k_n-2}, 2^{k_n}]$, we again distinguish between two cases:
	\begin{itemize}
	\item If $n \in (3 \cdot 2^{k_n-2}, 2^{k_n})$, then we have $n-1$ and $ n+1 \in [3 \cdot 2^{k_n-2}, 2^{k_n}]$ and by Proposition \ref{GFB_Decomposition}, Parts 2 and 3, $T_{n-1}^{gfb}, T_n^{gfb}$ and $T_{n+1}^{gfb}$ all contain a fully balanced tree of height $k_n-1$ as a maximal pending subtree.
	\item If $n = 2^{k_n}$, then $n-1 = 2^{k_n}-1$ and $n+1 = 2^{k_n}+1$. By Proposition 2, Part 2, $T_n^{gfb}$ and $T_{n-1}^{gfb}$ contain a fully balanced tree of height $k_n-1$ as a maximal pending subtree. Now, as $k_{n+1} = \lceil \log_2(n+1) \rceil = \lceil \log_2(2^{k_n}+1) \rceil = k_n+1$, by Proposition \ref{GFB_Decomposition}, Part 1, $T_{n+1}^{gfb}$ contains a fully balanced tree of height $(k_n+1)-2 = k_n-1$ as a maximal pending subtree. This completes the proof.
	\end{itemize}
\end{enumerate}
\end{proof}

We will now show that GFB trees always have minimal Colless index, i.e. we will prove the following theorem:
\begin{theorem} \label{GFB_is_minimal}
Let $T_n^{gfb}$ be a GFB tree with $n$ leaves. Then, $T_n^{gfb}$ has minimal Colless index, i.e. $\mathcal{C}(T_n^{gfb}) = c_n$.
\end{theorem}

In order to prove Theorem \ref{GFB_is_minimal}, we require the following technical lemma, the proof of which can be found in the Appendix.

\begin{lemma} \label{gfb_technical_lemma}
Let $s(x) = \min\limits_{z \in \mathbb{Z}}\vert  x-z \vert$, i.e. $s(x)$ is the distance from $x$ to the nearest integer. Let $n \in \mathbb{N}$ and let $k_n \coloneqq \lceil \log_2(n) \rceil$. Then,
	\begin{enumerate}
	\item For $n \in (2^{k_n-1}, 3 \cdot 2^{k_n-2}]$, we have
		\begin{align*}
		\frac{s(n \cdot 2^{1-k_n})}{2^{1-k_n}} &= n - 2^{k_n-1}.
		\end{align*}				
	\item For $n \in (3 \cdot 2^{k_n-2}, 2^{k_n}]$, we have
		\begin{align*}
		\frac{s(n \cdot 2^{1-k_n})}{2^{1-k_n}} &= 2^{k_n}-n.
		\end{align*}
	\end{enumerate}
\end{lemma}

We can now prove Theorem \ref{GFB_is_minimal}.
\begin{proof}[Proof of Theorem \ref{GFB_is_minimal}]
Let $T_n^{gfb}$ be a GFB tree with $n$ leaves. Let $k_n \coloneqq \lceil \log_2(n) \rceil $.
In order to show that $T_n^{gfb}$ has minimal Colless index, we show that $\mathcal{C}(T_n^{gfb}) = c_n = \sum\limits_{i=0}^{k_n-2} \frac{s(2^{i-k_n+1} \cdot n)}{2^{i-k_n+1}}$. We do this by induction on $n$. 

For $n=1$ we have $\mathcal{C}( T_1^{gfb})=0=c_1$, which gives the base case of the induction.

Now, we assume that the statement given in Theorem \ref{GFB_is_minimal} holds for all GFB trees with up to $n-1$ leaves and we show that it also holds for the GFB tree with $n$ leaves. 
 We now distinguish between 3 cases:
	\begin{itemize}
	\item $n = 3 \cdot 2^{k_n-2}$: \\
	By Proposition \ref{GFB_Decomposition} we know that $T_{3 \cdot 2^{k_n-2}}^{gfb}$ has the following standard decomposition: $T_{3 \cdot 2^{k_n-2}}^{gfb} = (T_{k_n-1}^{fb}, T_{k_n-2}^{fb})$. In particular, both maximal pending subtrees of $T_{3 \cdot 2^{k_n-2}}^{gfb}$ are fully balanced trees. By Theorem \ref{min_colless}, we have $\mathcal{C}(T_{k_n-1}^{fb})=0$ and $\mathcal{C}( T_{k_n-2}^{fb})=0$.
	Thus, using Lemma \ref{colless_sum}, we have
	\begin{equation} \label{interval_middle}
	\mathcal{C}(T_{3 \cdot 2^{k_n-2}}^{gfb}) = \underbrace{\mathcal{C}(T_{k_n-1}^{fb})}_{ = 0} + \underbrace{\mathcal{C}(T_{k_n-2}^{fb})}_{=0} +  2^{k_n-1} - 2^{k_n-2} = 2^{k_n-2}.
	\end{equation}
	On the other hand,
	\begin{align*}
	c_{3 \cdot 2^{k_n-2}} &= \sum\limits_{i=0}^{k_n-2} \frac{s(2^{i-k_n+1} \cdot (3 \cdot 2^{k_n-2}))}{2^{i-k_n+1}} \\
	&= \sum\limits_{i=0}^{k_n-2} \frac{s(3 \cdot 2^{i-1})}{2^{i-k_n+1}} \\
	&= \frac{s(3 \cdot 2^{0-1})}{2^{0-k_n+1}} \; \, \text{ as }3 \cdot 2^{i-1} \in \mathbb{Z} \text{ for } i > 0, \text{ and thus } s(3 \cdot 2^{i-1})=0 \text{ for } i > 0. \\
	&= \frac{s(\frac{3}{2})}{2^{1-k_n}} 
	= \frac{\frac{1}{2}}{2^{1-k_n}} 
	= \frac{1}{2} \cdot 2^{k_n-1} 
	= 2^{k_n-2} = \mathcal{C}(T_{3 \cdot 2^{k_n-2}}^{gfb}) \, \text{ by } \eqref{interval_middle}.
	\end{align*}
	Thus, $T_{3 \cdot 2^{k_n-2}}^{gfb}$ has minimal Colless index.
	\item $n \in (2^{k_n-1}, 3 \cdot 2^{k_n-2})$: \\
	By Proposition \ref{GFB_Decomposition} we know that $T_n^{gfb}$ has the following standard decomposition: $T_n^{gfb}=(T_{n-2^{k_n-2}},T_{k_n-2}^{fb})$. In particular, $T_{k_n-2}^{fb}$ is the fully balanced tree of height $k_n-2$, and thus by Theorem \ref{min_colless}, $\mathcal{C}(T_{k_n-2}^{fb})=0$. Moreover, for $T_{n-2^{k_n-2}}$ we have: $\lceil \log_2(n-2^{k_n-2}) \rceil = k_n-1$.
	Thus, using Lemma \ref{colless_sum} and Lemma \ref{max_subtrees}, we have
		\begin{align*}
		\mathcal{C}(T_n^{gfb}) &= \mathcal{C}(T_{n-2^{k_n-2}}) +  \underbrace{\mathcal{C}(T_{k_n-2}^{fb})}_{=0} +  (n - 2^{k_n-2}) - 2^{k_n-2} \\
		&= \mathcal{C}(T_{n-2^{k_n-2}}^{gfb}) + n - 2^{k_n-1} \; \text{ by Lemma \ref{gfb_subtrees}}\\
		&= c_{n-2^{k_n-2}} + n - 2^{k_n-1} \; \text{ by the inductive hypothesis} \\
		&= \sum\limits_{i=0}^{(k_n-1)-2} \frac{s(2^{i-(k_n-1)+1} \cdot (n - 2^{k_n-2}))}{2^{i-(k_n-1)+1}} + n - 2^{k_n-1} \; \text{ by Theorem \ref{colless_explicit}}\\
		&= \sum\limits_{i=0}^{k_n-3} \frac{s(2^{i-k_n+2} \cdot (n - 2^{k_n-2}))}{2^{i-k_n+2}} + n - 2^{k_n-1} \\
		&= \sum\limits_{i=0}^{k_n-3} \frac{s(2^{i-k_n+2} \cdot n -2^i)}{2^{i-k_n+2}}  + n - 2^{k_n-1} \\ 
		&= \sum\limits_{i=0}^{k_n-3} \frac{s(2^{i-k_n+2} \cdot n )}{2^{i-k_n+2}}  + n - 2^{k_n-1} \, \text{ by Lemma \ref{subbaditivity_of_s}, Part 2, as } 2^i \in \mathbb{Z} \\ 
		&= \frac{s(2^{2-k_n} \cdot n)}{2^{2-k_n}} + \frac{s(2^{3-k_n} \cdot n)}{2^{3-k_n}} + \ldots + \frac{s(2^{-1} \cdot n)}{2^{-1}} +n - 2^{k_n-1}\\
		&=  \frac{s(2^{1-k_n} \cdot n)}{2^{1-k_n}} + \frac{s(2^{2-k_n} \cdot n)}{2^{2-k_n}} + \frac{s(2^{3-k_n} \cdot n)}{2^{3-k_n}} + \ldots + \frac{s(2^{-1} \cdot n)}{2^{-1}}\\
		 &\text{ as } n - 2^{k_n-1} =  \frac{s(2^{1-k_n} \cdot n)}{2^{1-k_n}} \text{ by Lemma \ref{gfb_technical_lemma}, Part 1}. \\
		&= \sum\limits_{i=0}^{k_n-2} \frac{s(2^{i-k_n+1} \cdot n)}{2^{i-k_n+1}} 
		= c_n \, \text{ by Theorem \ref{colless_explicit}.}
		\end{align*}
		Thus, $T_n^{gfb}$ has minimal Colless index.
	\item $n \in (3 \cdot 2^{k_n-2}, 2^{k_n}]$: \\
	This case follows analogously to the previous case, but instead of using Part 1 of Lemma \ref{gfb_technical_lemma}, we use Part 2. For completeness the proof can be found in the Appendix.
	\end{itemize}
	Thus, in all cases, $\mathcal{C}(T_n^{gfb})=c_n$, which completes the proof. 
\end{proof}

Note that for $n=2^{k_n}$ Theorem \ref{GFB_is_minimal} together with Theorem \ref{min_colless} implies that $T^{gfb}_{2^{k_n}} = T^{fb}_{k_n}$.

\subsubsection{Further characterizing and counting trees with minimal Colless index}\label{sec_characterization}
So far, we have seen that there are two classes of trees with minimal Colless index, namely maximally balanced trees and GFB trees. However, there are trees that are neither maximally balanced nor GFB, but still have minimal Colless index (e.g. tree $T_1$ depicted in Figure \ref{Fig_Sackin_not_Colless}). 
In the following, we will thus try to further characterize and count trees with minimal Colless index. 

In particular, we will show that the leaf partitioning of $n$ leaves into $n_a$ and $n_b$ as induced by Algorithm \ref{alg_gfb} is the most extreme one a tree with minimal Colless index can have. This means that given a tree $T$ with minimal Colless index, the difference in the number of leaves of its two maximal pending subtrees $n_a - n_b$ cannot be larger than it is in a GFB tree. To be precise, we have the following theorem:

\begin{figure}[htbp]
\centering
\includegraphics[scale=0.45]{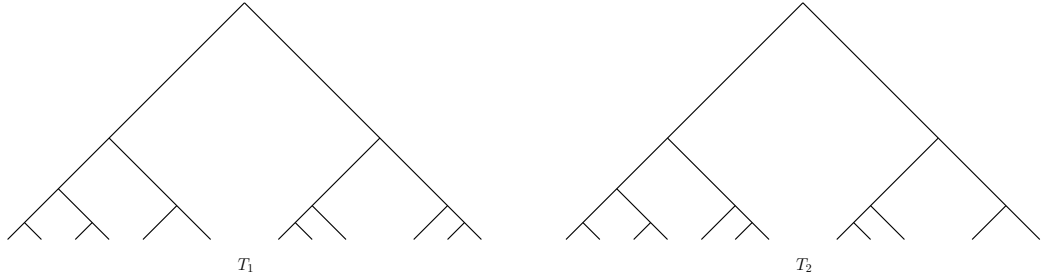}
\caption{Trees $T_1$ and $T_2$ on 12 leaves. We have $\mathcal{C}(T_1)=4=c_{12}$ and $\mathcal{C}(T_2) = 6$. Thus, $T_1$ has minimal Colless index, while $T_2$ does not (this due to Theorem \ref{oddodd}). Note, however, that for the Sackin index (cf. Definition \ref{Def_Sackin} on page \pageref{Def_Sackin}) we have $\mathcal{S}(T_1) = \mathcal{S}(T_2) = 44$, which can be shown to be minimal (cf. Theorem 3 in \citet{Fischer2018}).}
\label{Fig_Sackin_not_Colless}
\end{figure}

\begin{theorem} \label{gfb_extreme}
Let $T^{gfb}_n = (T_a, T_b)$ be a GFB tree on $n$ leaves with $n \in (2^{k_n-1}, 2^{k_n})$ (where $k_n = \lceil \log_2(n) \rceil$), i.e.
$$ (n_a, n_b) = 
	\begin{cases}
	(n - 2^{k_n-2}, 2^{k_n-2}), &\text{ if } n \in (2^{k_n-1}, 3 \cdot 2^{k_n-2}); \\
	(2^{k_n-1}, 2^{k_n-2}), &\text{ if } n = 3 \cdot 2^{k_n-2}; \\
	(2^{k_n-1}, n-2^{k_n-1}), &\text{ if } n \in (3 \cdot 2^{k_n-2}, 2^{k_n}),
	\end{cases}	
	$$
where we have $n_a-n_b = n-2^{k_n-1}$ in the first case, $n_a-n_b=2^{k_n-2}$ in the second case and $n_a-n_b=2^{k_n}-n$ in the last case. Now, suppose that $\widehat{T} = (\widehat{T}_a, \widehat{T}_b)$ is a tree with a more extreme leaf partitioning, i.e. $\widehat{n}_a - \widehat{n}_b > n_a-n_b$ or, to be more precise,
$$ (\widehat{n}_a, \widehat{n}_b) = 
	\begin{cases}
	(n-2^{k_n-2}+j, 2^{k_n-2}-j) \text{ with } j \in \{1, \ldots, 2^{k_n-2}-1\}, &\text{ if } n \in (2^{k_n-1}, 3 \cdot 2^{k_n-2}); \\
	(2^{k_n-1}+j, 2^{k_n-2}-j) \text{ with } j \in \{1, \ldots, 2^{k_n-2}-1 \}, &\text{ if } n = 3 \cdot 2^{k_n-2}; \\
	(2^{k_n-1} + j, n - 2^{k_n-1} - j) \text{ with } j \in \{1, \ldots, n-2^{k_n-1}-1\}, &\text{ if }  n \in (3 \cdot 2^{k_n-2}, 2^{k_n}).
	\end{cases}$$
Then, we have
$$ c_n = \mathcal{C}(T^{gfb}_n) < \mathcal{C}(\widehat{T}),$$
i.e. $\widehat{T}$ does \emph{not} have minimal Colless index (where the first equality follows from Theorem \ref{GFB_is_minimal}). 
\end{theorem}

The proof of this theorem requires the following lemma, which provides an upper bound on the maximal minimal Colless index for any given $n \in (2^{k_n-1}, 2^{k_n})$. This upper bound is also depicted in Figure \ref{Fig_MinimalColless}.

\begin{lemma} \label{max_min_Colless}
Let $n \in (2^{k_n-1}, 2^{k_n})$ with $k_n = \lceil \log_2(n) \rceil$. Let $\bar{c}_n$ denote the maximal minimal Colless index for $n \in (2^{k_n-1}, 2^{k_n})$, i.e. 
$$ \bar{c}_n = \max\limits_{n \in (2^{k_n-1}, 2^{k_n})} c_n.$$
Then, we have
$$ \bar{c}_n < 2^{k_n-1}.$$ 
\end{lemma} 

\begin{proof}
By Theorem \ref{colless_explicit}, we have for the minimal Colless index
\begin{align*}
c_n &= \sum\limits_{i=0}^{k_n-2} \frac{s(2^{i-k_n+1} \cdot n)}{2^{i-k_n+1}} \\
&= \frac{s(2^{1-k_n} \cdot n)}{2^{1-k_n}} + \frac{s(2^{2-k_n} \cdot n)}{2^{2-k_n}} + \ldots + \frac{s(2^{-2} \cdot n)}{2^{-2}} + \frac{s(2^{-1} \cdot n)}{2^{-1}} \\
&\leq \frac{\frac{1}{2}}{2^{1-k_n}} + \frac{\frac{1}{2}}{2^{2-k_n}} + \ldots + \frac{\frac{1}{2}}{2^{-2}} + \frac{\frac{1}{2}}{2^{-1}}, \, \text{ as } s(x) = \min\limits_{z \in \mathbb{Z}}\vert  x-z \vert \leq \frac{1}{2} \, \forall x \\
&= \sum\limits_{m=0}^{k_n-2} 2^m 
= \sum\limits_{m=0}^{k_n} 2^m - 2^{k_n-1} - 2^{k_n} 
= (2^{k_n+1} - 1) - 2^{k_n-1} - 2^{k_n} \\
&= 2^{k_n-1} - 1 < 2^{k_n-1}.
\end{align*}
This means that $c_n$ is bounded from above by $2^{k_n-1}$. 
In particular, $\bar{c}_n = \max\limits_{n \in (2^{k_n-1}, 2^{k_n})} c_n < 2^{k_n-1}$, which completes the proof. 
\end{proof}

We are now in a position to prove Theorem \ref{gfb_extreme}, which is divided into three subcases. For $n = 3 \cdot 2^{k_n-2}$ the proof is straightforward, while the cases $n \in (2^{k_n-1}, 3 \cdot 2^{k_n-2})$ and $n \in (3 \cdot 2^{k_n-2}, 2^{k_n})$ are more technical and require Lemma \ref{subbaditivity_of_s} and Lemma \ref{max_min_Colless}. 
\begin{proof}[Proof of Theorem \ref{gfb_extreme}]
Let $n \in (2^{k_n-1}, 2^{k_n})$ with $k_n = \lceil \log_2(n) \rceil$. Let $T^{gfb}_n=(T_a,T_b)$ be the GFB tree on $n$ leaves. We now distinguish between three cases (only case 1. is given here; the other two cases are shown in the Appendix):
\begin{enumerate}
\item $n \in (2^{k_n-1}, 3 \cdot 2^{k_n-2})$: \\
From Proposition \ref{GFB_Decomposition}, we have for $T^{gfb}_n=(T_a,T_b)$ that $n_a=n-2^{k_n-2}$ and $n_b=2^{k_n-2}$. In particular, $T_b$ is the fully balanced tree of height $k_n-2$ and we have $\mathcal{C}(T_b)=c_{n_b}=0$ by Theorem \ref{min_colless}. For $T_a$ we have $k_{n_a} = \lceil \log_2(n_a) \rceil = k_n-1$.  
Now, consider $\mathcal{C}(T^{gfb}_n)$. By Lemmas \ref{colless_sum} and \ref{max_subtrees} and by Theorem \ref{GFB_is_minimal} we have
\begin{align}
\mathcal{C}(T^{gfb}_n) &= n_a - n_b + c_{n_a} + c_{n_b} \nonumber \\
&= n_a - n_b + c_{n_a} \label{Colless_T_cna} \\
&= n_a - n_b + c_{n-2^{k_n-2}} \nonumber \\
&= n_a - n_b + \sum\limits_{i=0}^{k_n-3} \frac{s(2^{i-k_n+2} \cdot (n-2^{k_n-2}))}{2^{i-k_n+2}} \nonumber \\
&= n_a - n_b + \sum\limits_{i=0}^{k_n-3} \frac{s(2^{i-k_n+2} \cdot n - 2^i)}{2^{i-k_n+2}} \nonumber \\
&= n_a - n_b + \sum\limits_{i=0}^{k_n-3} \frac{s(2^{i-k_n+2} \cdot n)}{2^{i-k_n+2}} \, \text{ by Lemma \ref{subbaditivity_of_s}, Part 2, as } 2^i \in \mathbb{Z} \nonumber \\
&= n_a - n_b + \sum\limits_{i=0}^{k_n-3} \frac{s(2^{i-k_n+2} \cdot ((n+j)-j))}{2^{i-k_n+2}} \nonumber \\
&= n_a - n_b + \sum\limits_{i=0}^{k_n-3} \frac{s(2^{i-k_n+2} \cdot (n+j) + 2^{i-k_n+2} \cdot (-j)) }{2^{i-k_n+2}}. \label{Colless_T_firsthalf}
\end{align}
Now, suppose that $\widehat{T} = (\widehat{T}_a, \widehat{T}_b)$ with $\widehat{n}_a = n_a+j$ and $\widehat{n}_b=n_b-j$ (and $j \in \{1, \ldots, 2^{k_n-2}-1\}$) is also a tree with minimal Colless index, i.e. $\mathcal{C}(\widehat{T}) = \mathcal{C}(T^{gfb}_n)=c_n$. Consider $\mathcal{C}(\widehat{T})$. Again, by Lemmas \ref{colless_sum} and \ref{max_subtrees}, we have
\begin{align}
\mathcal{C}(\widehat{T}) &= \widehat{n}_a - \widehat{n}_b + c_{\widehat{n}_a} + c_{\widehat{n}_b} \nonumber \\
&= (n_a+j) - (n_b-j) + c_{\widehat{n}_a} + c_{\widehat{n}_b} \nonumber \\ 
&= n_a - n_b + 2j + c_{\widehat{n}_a}+ c_{\widehat{n}_b} \nonumber \\
&= n_a - n_b + 2j + \underbrace{c_{n-2^{k_n-2}+j}}_{\geq 0} + \underbrace{c_{2^{k_n-2}-j}}_{\geq 0}. \label{Colless_That_firsthalf}
\end{align}
Now, comparing $\mathcal{C}(T^{gfb}_n)$ (see Equation \eqref{Colless_T_cna}) and $\mathcal{C}(\widehat{T})$ (see Equation \eqref{Colless_That_firsthalf}) it immediately follows that we have $\mathcal{C}(T^{gfb}_n) < \mathcal{C}(\widehat{T})$ if $c_{n_a} < 2j$, which would contradict the minimality of $\mathcal{C}(\widehat{T})$. Thus, we have $c_{n_a} \geq 2j$. We now claim that $c_{n_a} \geq 2j$ implies $k_{\widehat{n}_b} = \lceil \log_2 (\widehat{n}_b) \rceil =k_{n_b} = k_n-2$. To see this, consider the following:
	\begin{itemize}
	\item As $\widehat{n}_b < n_b$, we have $k_{\widehat{n}_b} \leq k_{n_b} = k_n-2$.
	\item Suppose $k_{\widehat{n}_b} \leq k_n-3$. As $\widehat{n}_b = n_b-j = 2^{k_n-2}-j$, this implies $j \geq 2^{k_n-3}$. In particular, $2j \geq 2^{k_n-2}$. \\
	Now, consider $c_{n_a}$. As $k_{n_a}=k_n-1$ it follows from Lemma \ref{max_min_Colless} that
	\begin{align*}
	c_{n_a} < 2^{(k_n-1)-1} = 2^{k_n-2}.
	\end{align*}
	However, as $2j \geq 2^{k_n-2}$, this implies $c_{n_a} < 2j$, which contradicts the assumption that $c_{n_a} \geq 2j$. Thus, $k_{\widehat{n}_b} > k_n-3$.
	\end{itemize}
Thus, in total we have $k_{\widehat{n_b}} = k_n-2$.
This, however, implies that 
$$ j \leq 2^{k_n-3} - 1 < 2^{k_n-3} = \frac{1}{2} \cdot 2^{k_n-2}.$$
In particular, $j \in (0 \cdot 2^{k_n-2}, \frac{1}{2} \cdot 2^{k_n-2})$ and $2^{2-k_n} \cdot j \in (0, \frac{1}{2})$.
This, in turn, implies
	\begin{align}
	\frac{s(2^{2-k_n} \cdot (-j))}{2^{2-k_n}} &= \frac{s(2^{2-k_n} \cdot j)}{2^{2-k_n}} = \frac{2^{2-k_n} \cdot j}{2^{2-k_n}} = j, \label{j}
	\end{align}
a fact that will be used later on.\\

Moreover, we have for $\widehat{T}_a$: $k_{\widehat{n}_a} = \lceil \log_2(\widehat{n}_a) \rceil \in \{k_n-1, k_n\}$:
	\begin{itemize}
	\item As $\widehat{n}_a > n_a$ and $k_{n_a} = k_n-1$, it follows that $k_{\widehat{n}_a} \geq k_n-1$. 
	\item However, as $\widehat{n}_a < n$, it also follows that $k_{\widehat{n}_a} \leq k_n$.
	\end{itemize}
In total $k_{\widehat{n}_a} \in \{k_n-1, k_n\}$. \\

Thus, for $\widehat{T}$ we now distinguish between two cases:
	\begin{enumerate}
	\item $k_{\widehat{n}_a} = k_n-1$ and $k_{\widehat{n}_b} = k_n-2$: \\
		Continuing from Equation \eqref{Colless_That_firsthalf} and using Theorem \ref{colless_explicit}, we have
			\begin{align}
			\mathcal{C}(\widehat{T}) &= n_a - n_b + 2j + c_{n-2^{k_n-2}+j} + c_{2^{k_n-2}-j} \nonumber \\
			&= n_a - n_b  + 2j + \sum\limits_{i=0}^{k_n-3} \frac{s(2^{i-k_n+2} \cdot (n-2^{k_n-2}+j))}{2^{i-k_n+2}} + \sum\limits_{i=0}^{k_n-4} \frac{s(2^{i-k_n+3} \cdot (2^{k_n-2}-j))}{2^{i-k_n+3}} \nonumber \\
			&= n_a - n_b + 2j + \sum\limits_{i=0}^{k_n-3} \frac{s(2^{i-k_n+2} \cdot (n+j) - 2^i)}{2^{i-k_n+2}} + \sum\limits_{i=0}^{k_n-4} \frac{s( 2^{i+1} - 2^{i-k_n+3} \cdot j)}{2^{i-k_n+3}}  \nonumber \\
			&\text{ by Lemma \ref{subbaditivity_of_s}, Part 2, as } 2^{i} \text{ and } 2^{i+1} \in \mathbb{Z} \nonumber \\
			&= n_a - n_b + 2j + \sum\limits_{i=0}^{k_n-3} \frac{s(2^{i-k_n+2} \cdot (n+j))}{2^{i-k_n+2}} + \sum\limits_{i=0}^{k_n-4} \frac{s(2^{i-k_n+3} \cdot (-j))}{2^{i-k_n+3}}  \nonumber  \\
			&= n_a - n_b + 2j + \sum\limits_{i=0}^{k_n-3} \frac{s(2^{i-k_n+2} \cdot (n+j))}{2^{i-k_n+2}} + \sum\limits_{i=1}^{k_n-3} \frac{s(2^{i-k_n+2} \cdot (-j))}{2^{i-k_n+2}} \nonumber  \\
			&= n_a - n_b + 2j + \sum\limits_{i=0}^{k_n-3} \frac{s(2^{i-k_n+2} \cdot (n+j))}{2^{i-k_n+2}} + \sum\limits_{i=0}^{k_n-3} \frac{s(2^{i-k_n+2} \cdot (-j))}{2^{i-k_n+2}} - \frac{s(2^{2-k_n} \cdot (-j))}{2^{2-k_n}} \nonumber \\
	  &= n_a - n_b + 2j + \sum\limits_{i=0}^{k_n-3} \frac{s(2^{i-k_n+2} \cdot (n+j))}{2^{i-k_n+2}} + \sum\limits_{i=0}^{k_n-3} \frac{s(2^{i-k_n+2} \cdot (-j))}{2^{i-k_n+2}} - j \, \text{ by } \eqref{j} \nonumber \\
		&= n_a - n_b + j + \sum\limits_{i=0}^{k_n-3} \frac{s(2^{i-k_n+2} \cdot (n+j))}{2^{i-k_n+2}} + \sum\limits_{i=0}^{k_n-3} \frac{s(2^{i-k_n+2} \cdot (-j))}{2^{i-k_n+2}}. \label{Colless_That_first1_final}
		\end{align}
		Now, as by assumption both $T^{gfb}_n$ and $\widehat{T}$ are trees with minimal Colless index, we have (using Equations \eqref{Colless_T_firsthalf} and \eqref{Colless_That_first1_final})
		\begin{align*}
		0 &= \mathcal{C}(T^{gfb}_n) - \mathcal{C}(\widehat{T}) \\
		&= n_a - n_b + \sum\limits_{i=0}^{k_n-3} \frac{s(2^{i-k_n+2} \cdot (n+j) + 2^{i-k_n+2} \cdot (-j)) }{2^{i-k_n+2}} \\
			&\hspace*{5mm} -n_a + n_b - j - \sum\limits_{i=0}^{k_n-3} \frac{s(2^{i-k_n+2} \cdot (n+j))}{2^{i-k_n+2}} - \sum\limits_{i=0}^{k_n-3} \frac{s(2^{i-k_n+2} \cdot (-j))}{2^{i-k_n+2}}  \\
		&= \sum\limits_{i=0}^{k_n-3} \underbrace{\frac{s(2^{i-k_n+2} \cdot (n+j) + 2^{i-k_n+2} \cdot (-j)) - \big( s (2^{i-k_n+2} \cdot (n+j)) + s(2^{i-k_n+2} \cdot (-j)) \big)}{2^{i-k_n+2}}}_{\leq 0 \text{ by Lemma  \ref{subbaditivity_of_s}, Part 1}} - j \\
		&< 0.
		\end{align*}
		This, however, is a contradiction. Thus, $\mathcal{C}(\widehat{T})$ is not minimal, which completes the proof for this subcase. 
	\item $k_{\widehat{n}_a} = k_n$ and $k_{\widehat{n}_b} = k_n-2$: \\
	Again, continuing from Equation \eqref{Colless_That_firsthalf} and using Theorem \ref{colless_explicit}, we have
	\begin{align}
	\mathcal{C}(\widehat{T}) &= n_a - n_b + 2j + c_{n-2^{k_n-2}+j} + c_{2^{k_n-2}-j} \nonumber \\
	&= n_a - n_b + 2j + \sum\limits_{i=0}^{k_n-2} \frac{s(2^{i-k_n+1} \cdot (n-2^{k_n-2}+j))}{2^{i-k_n+1}} + \sum\limits_{i=0}^{k_n-4} \frac{s(2^{i-k_n+3} \cdot (2^{k_n-2} - j))}{2^{i-k_n+3}} \nonumber \\
	&= n_a - n_b + 2j + \sum\limits_{i=0}^{k_n-2} \frac{s(2^{i-k_n+1} \cdot (n+j) - 2^{i-1})}{2^{i-k_n+1}} + \sum\limits_{i=0}^{k_n-4} \frac{s( 2^{i+1} + 2^{i-k_n+3} \cdot (-j))}{2^{i-k_n+3}} \nonumber \\
	&= n_a -n_b + 2j + \sum\limits_{i=0}^{k_n-2} \frac{s(2^{i-k_n+1} \cdot (n+j) - 2^{i-1})}{2^{i-k_n+1}} + \sum\limits_{i=0}^{k_n-4} \frac{s(2^{i-k_n+3} \cdot (-j))}{2^{i-k_n+3}} \nonumber \\
	&\text{ by Lemma \ref{subbaditivity_of_s}, Part 2, as } 2^{i+1} \in \mathbb{Z} \nonumber \\
	&= n_a - n_b + 2j + \frac{s(2^{1-k_n} \cdot (n+j) - 2^{-1})}{2^{1-k_n}} + \sum\limits_{i=1}^{k_n-2} \frac{s(2^{i-k_n+1} \cdot (n+j) - 2^{i-1})}{2^{i-k_n+1}} \nonumber \\
	&\hspace*{5mm} + \sum\limits_{i=0}^{k_n-4} \frac{s(2^{i-k_n+3} \cdot (-j))}{2^{i-k_n+3}}  \nonumber \\
	&= n_a - n_b + 2j + \frac{s(2^{1-k_n} \cdot (n+j) - 2^{-1})}{2^{1-k_n}} + \sum\limits_{i=0}^{k_n-3} \frac{s(2^{i-k_n+2} \cdot (n+j) - 2^i)}{2^{i-k_n+2}} \nonumber \\
	&\hspace*{5mm} + \sum\limits_{i=0}^{k_n-4} \frac{s(2^{i-k_n+3} \cdot (-j))}{2^{i-k_n+3}} \nonumber \\
	&= n_a - n_b + 2j + \frac{s(2^{1-k_n} \cdot (n+j) - 2^{-1})}{2^{1-k_n}} + \sum\limits_{i=0}^{k_n-3} \frac{s(2^{i-k_n+2} \cdot (n+j))}{2^{i-k_n+2}} \nonumber \\
	&\hspace*{5mm} + \sum\limits_{i=0}^{k_n-4} \frac{s(2^{i-k_n+3} \cdot (-j))}{2^{i-k_n+3}} \text{ by Lemma \ref{subbaditivity_of_s}, Part 2, as } 2^{i} \in \mathbb{Z} \nonumber \\
	&= n_a - n_b + 2j + \frac{s(2^{1-k_n} \cdot (n+j) - 2^{-1})}{2^{1-k_n}} + \sum\limits_{i=0}^{k_n-3} \frac{s(2^{i-k_n+2} \cdot (n+j))}{2^{i-k_n+2}} \nonumber \\
	&\hspace*{5mm} + \sum\limits_{i=1}^{k_n-3} \frac{s(2^{i-k_n+2} \cdot (-j))}{2^{i-k_n+2}}  \nonumber \\
	&= n_a - n_b + 2j + \frac{s(2^{1-k_n} \cdot (n+j) - 2^{-1})}{2^{1-k_n}} + \sum\limits_{i=0}^{k_n-3} \frac{s(2^{i-k_n+2} \cdot (n+j))}{2^{i-k_n+2}} \nonumber \\
	&\hspace*{5mm} + \sum\limits_{i=0}^{k_n-3} \frac{s(2^{i-k_n+2} \cdot (-j))}{2^{i-k_n+2}} -\frac{s(2^{2-k_n} \cdot (-j))}{2^{2-k_n}} \nonumber \\
	&= n_a - n_b + 2j + \frac{s(2^{1-k_n} \cdot (n+j) - 2^{-1})}{2^{1-k_n}} + \sum\limits_{i=0}^{k_n-3} \frac{s(2^{i-k_n+2} \cdot (n+j))}{2^{i-k_n+2}} \nonumber \\
	&\hspace*{5mm} + \sum\limits_{i=0}^{k_n-3} \frac{s(2^{i-k_n+2} \cdot (-j))}{2^{i-k_n+2}} - j \, \text{ by } \eqref{j} . \label{Colless_That_first2_final}
	\end{align}
	Again, as by assumption both $T^{gfb}_n$ and $\widehat{T}$ are trees with minimal Colless index, we have (using Equations \eqref{Colless_T_firsthalf} and \eqref{Colless_That_first2_final})
	\begin{align*}
	0 &= \mathcal{C}(T^{gfb}_n) - \mathcal{C}(\widehat{T}) \\
	&= n_a - n_b + \sum\limits_{i=0}^{k_n-3} \frac{s(2^{i-k_n+2} \cdot (n+j) + 2^{i-k_n+2} \cdot (-j)) }{2^{i-k_n+2}} \\
			&\hspace*{5mm} -n_a + n_b - j - \frac{s(2^{1-k_n} \cdot (n+j) - 2^{-1})}{2^{1-k_n}} - \sum\limits_{i=0}^{k_n-3} \frac{s(2^{i-k_n+2} \cdot (n+j))}{2^{i-k_n+2}} \\
	&\hspace*{5mm} - \sum\limits_{i=0}^{k_n-3} \frac{s(2^{i-k_n+2} \cdot (-j))}{2^{i-k_n+2}} \\
	&= \sum\limits_{i=0}^{k_n-3} \underbrace{\frac{s(2^{i-k_n+2} \cdot (n+j) + 2^{i-k_n+2} \cdot (-j)) - \big( s (2^{i-k_n+2} \cdot (n+j)) + s(2^{i-k_n+2} \cdot (-j)) \big)}{2^{i-k_n+2}}}_{\leq 0 \text{ by Lemma  \ref{subbaditivity_of_s}, Part 1}} \\
	&\hspace*{5mm} - \underbrace{j}_{> 0} - \underbrace{\frac{s(2^{1-k_n} \cdot (n+j) - 2^{-1})}{2^{1-k_n}}}_{\geq 0} \\
	&< 0.
	\end{align*}
	This, however, is a contradiction. Thus, $\mathcal{C}(\widehat{T})$ is not minimal, which completes the proof for this subcase. 
	\end{enumerate}
\end{enumerate}
For $n = 3 \cdot 2^{k_n-2}$ the proof is straightforward and for $n \in (3 \cdot 2^{k_n-2}, 2^{k_n})$ the proof is similar to the case shown above. Thus, both cases are given in the Appendix.

In all cases, we have $\mathcal{C}(\widehat{T}) > \mathcal{C}(T^{gfb}_n)$, which completes the proof.
\end{proof}

To summarize, we have seen that both maximally balanced trees and GFB trees have minimal Colless index (cf. Theorems \ref{colless_minimum_mb} and \ref{GFB_is_minimal}). 
Moreover, for a tree $T=(T_a, T_b)$ with minimal Colless index we always have:

\begin{corollary} \label{leaf_partioning}
Let $T=(T_a, T_b)$ be a tree on $n$ leaves with minimal Colless index, i.e. $\mathcal{C}(T)=c_n$. Let $n_a, n_b$ denote the number of leaves of $T_a$ and $T_b$, respectively, where $n_a \geq n_b$. Then
\begin{align*}
n_a^{gfb} &\geq n_a \geq n_a^{mb} \\
n_b^{gfb} &\leq n_b \leq n_b^{mb},
\end{align*}
where $n_a^{gfb}$ and $n_b^{gfb}$ denote the number of leaves of $T^{gfb}_a$ and $T^{gfb}_b$ in $T_n^{gfb} = (T_a^{gfb}, T_b^{gfb})$ and $n_a^{mb}$ and $n_b^{mb}$ denote the number of leaves of $T_a^{mb}$ and $T_b^{mb}$ in $T_n^{mb}=(T_a^{mb}, T_b^{mb})$.  
\end{corollary}

\begin{proof}
Let $T=(T_a, T_b)$ be a tree on $n$ leaves with minimal Colless index, i.e. $\mathcal{C}(T)=c_n$. Let $n_a, n_b$ denote the number of leaves of $T_a$ and $T_b$, respectively, where $n_a \geq n_b$. 

Now, recall that both maximally balanced trees and GFB trees have minimal Colless index (cf. Theorems \ref{colless_minimum_mb} and \ref{GFB_is_minimal}).

Assume $n_a < n_a^{mb}$ and $n_b > n_b^{mb}$. As $n_a^{mb} = n_b^{mb} = \frac{n}{2}$ for $n$ even and $n_a^{mb} = \frac{n+1}{2}$ and $n_b^{mb} = \frac{n-1}{2}$ for $n$ odd, this assumption contradicts $n_a \geq n_b$. Thus, $n_a \geq n_a^{mb}$ and $n_b \leq n_b^{mb}$. 

Additionally, $n_a \leq n_a^{gfb}$ and $n_b \geq n_b^{gfb}$ is a direct consequence of Theorem \ref{gfb_extreme}. 
This completes the proof.
\end{proof}

Note that Corollary \ref{leaf_partioning} only gives a necessary and not a sufficient condition. Consider for example tree $T_2$ in Figure \ref{Fig_Sackin_not_Colless} on $12$ leaves. Here, $n_a=7$ and $n_b=5$, i.e.
\begin{align*}
n_a^{gfb} = 8 &\geq 7 \geq 6 = n_a^{mb} \\
n_b^{gfb} = 4 &\leq 5 \leq 6 = n_b^{mb},
\end{align*}
but $T_2$ does not have minimal Colless index. 
This is due to the fact that if $n_a \neq n_b$ and $n_b, n_b$ odd, the resulting tree will not have minimal Colless index:

\begin{theorem}\label{oddodd}
Let $T = (T_a,T_b)$ be a tree on $n$ leaves and let $n_a$ and $n_b$ denote the number of leaves of $T_a$ and $T_b$. If $n_a \neq n_b$ and $n_a, n_b$ odd, then $C(T) > c_n$, i.e. $T$ does not have minimal Colless index.
\end{theorem}
\begin{proof}
Let $T = (T_a,T_b)$ be a tree on $n$ leaves. Let $n_a$ and $n_b$ denote the number of leaves of $T_a$ and $T_b$ with $n_a \neq n_b$ and $n_a, n_b$ odd. Without loss of generality let $n_a > n_b$. The fact that $n_a$ and $n_b$ are both odd, results in 
\begin{align}
n_a &\geq n_b +2 \nonumber \\
\Leftrightarrow \frac{n_a -1}{2} &\geq \frac{n_b +1}{2}. \label{nanbgeq}
\end{align}
By \eqref{nanbgeq} we also have
\begin{align}
\frac{n_a +1}{2} &> \frac{n_b +1}{2}. \label{nanbgeq2}
\end{align}
We prove the statement by contradiction and assume that $C(T) = c_n$, i.e. $T$ has minimal Colless index. Then,
\begin{align}
c_n &= c_{n_a} + c_{n_b} + n_a - n_b \quad \text{by Lemma } \ref{colless_sum} \text{ and Lemma } \ref{max_subtrees} \nonumber \\
&= c_{\frac{n_a+1}{2}} + c_{\frac{n_a-1}{2}} + 1 + c_{\frac{n_b+1}{2}} + c_{\frac{n_b-1}{2}} + 1 + n_a - n_b \quad \text{by Theorem } \ref{colless_minimum}\nonumber \\
&= c_{\frac{n_a+1}{2}} + c_{\frac{n_a-1}{2}} + c_{\frac{n_b+1}{2}} + c_{\frac{n_b-1}{2}} + n_a - n_b + 2. \label{nanboddumformung}
\end{align}
Additionally, by Lemma \ref{colless_sum}, Lemma \ref{max_subtrees} and \eqref{nanbgeq2}
\begin{align*}
c_{\frac{n}{2}} = c_{\frac{n_a+n_b}{2}} \leq c_{\frac{n_a+1}{2}} + c_{\frac{n_b-1}{2}} + \frac{n_a+1}{2} - \frac{n_b-1}{2},
\end{align*}
which results in
\begin{align}
c_{\frac{n}{2}} - \frac{n_a-n_b}{2} - 1 \leq c_{\frac{n_a+1}{2}} + c_{\frac{n_b-1}{2}}. \label{term1}
\end{align}
Similarly, by Lemma \ref{colless_sum}, Lemma \ref{max_subtrees} and \eqref{nanbgeq}
\begin{align*}
c_{\frac{n}{2}} = c_{\frac{n_a+n_b}{2}} \leq c_{\frac{n_a-1}{2}} + c_{\frac{n_b+1}{2}} + \frac{n_a-1}{2} - \frac{n_b+1}{2},
\end{align*}
which results in
\begin{align}
c_{\frac{n}{2}} - \frac{n_a-n_b}{2} + 1 \leq c_{\frac{n_a-1}{2}} + c_{\frac{n_b+1}{2}}. \label{term2}
\end{align}
By \eqref{term1} and \eqref{term2} and the fact that $n$ is even we can rewrite \eqref{nanboddumformung} in the following way
\begin{align*}
c_n &= c_{\frac{n_a+1}{2}} + c_{\frac{n_a-1}{2}} + c_{\frac{n_b+1}{2}} + c_{\frac{n_b-1}{2}} + n_a - n_b + 2\\
&\geq c_{\frac{n}{2}} - \frac{n_a-n_b}{2} - 1 + c_{\frac{n}{2}} - \frac{n_a-n_b}{2} + 1 + n_a - n_b +2 \\
&= 2 c_{\frac{n}{2}} + 2 \\
&= c_n + 2 \quad \text{by Theorem } \ref{colless_minimum}.
\end{align*}
This results in $c_n \geq c_n + 2$ which is a contradiction, and therefore completes the proof.
\end{proof}

Thus, if $n_a \neq n_b$, and $n_a, n_b$ odd, the resulting tree does not have minimal Colless index.
However, it can easily be verified that if $n_a \neq n_b$, but both $n_a$ and $n_b$ are even, the resulting tree may or may not have minimal Colless index.
Consider for example $n=24$: In this case $\lceil \log_2(24) \rceil = 5$ and thus by Theorem \ref{colless_explicit} we have:
\begin{align*}
c_{24} &= \sum\limits_{i=0}^3 \frac{s(2^{i-4} \cdot 24)}{2^{i-4}} = 8 \\
&= c_{16} + c_8 + 8  = 0 + 0 + 8\\
&< c_{14} + c_{10} + 4 = 4 + 4+ 4 = 12.
\end{align*}
Thus, while $n_a=16, n_b=8$ yields a Colless minimum, $n_a=14, n_b=10$ does not.

We now turn to the number of trees on $n$ leaves with minimal Colless index and state an upper bound for it. This bound is implied by the fact that we can relate trees with minimal Colless index with trees with minimal Sackin index, another index of tree balance which is defined as follows:
\begin{definition}[\citet{Sackin1972}]  \label{Def_Sackin}
Let $T$ be a rooted binary tree. Then, its \emph{Sackin index} is defined as 
$$ \mathcal{S}(T) = \sum\limits_{u \, \in \, \mathring{V}(T)} n_u,$$
where $n_u$ denotes the number of leaves of the subtree of $T$ rooted at $u$.
\end{definition}

For the Sackin index, the number of trees with minimal Sackin index is known (cf. Theorem 8 in \citet{Fischer2018}). 

\begin{theorem}[adapted from Theorem 8 in \citet{Fischer2018}]
\label{recursion} Let $\widetilde{s}(n)$ denote the number of binary rooted trees with $n$ leaves and with minimal Sackin index and let $k_n \coloneqq \lceil \log_2(n)\rceil$. For any partition of $n$ into two integers $n_a$, $n_b$, i.e. $n=n_a+n_b$, we use $k_{n_a}$ and $k_{n_b}$ to denote $\lceil \log_2(n_a)\rceil$ and $\lceil \log_2(n_b)\rceil=\lceil \log_2(n-n_a)\rceil$, respectively. Moreover, let $$\widetilde{f}(n)=\begin{cases}0 & \mbox{if $n$ is odd} \\ {\widetilde{s}(\frac{n}{2})+1 \choose 2}& \mbox{else. }\end{cases}$$  Then, the following recursion holds: 
\begin{itemize}
\item $\widetilde{s}(1)=1$
\item $\widetilde{s}(n)=  \sum\limits_{\substack{(n_a,n_b): \\n_a+n_b=n,\\ n_a\geq \frac{n}{2},\\ k_{n_a} = k -1, \\ k_{n_b} = k -1, \\n_a \neq n_b }} \left( \widetilde{s}(n_a)\cdot \widetilde{s}(n_b) \right) +\widetilde{f}(n)+\widetilde{s}(n-2^{k-2}).$
\end{itemize}
\end{theorem}

Note that above recursion yields sequence \texttt{A299037} in the On-Line Encyclopedia of Integer Sequences (\citet{OEIS}).

In the following we will show that every tree with minimal Colless index also has minimal Sackin index. 
Note, however, that the converse is not true: tree $T_2$ depicted in Figure \ref{Fig_Sackin_not_Colless} has minimal Sackin index, but does not have minimal Colless index. 
Nevertheless, as we will show next, the number of trees with minimal Sackin index on $n$ leaves provides an upper bound for the number of leaves with minimal Colless index.

\begin{proposition} \label{Colless_implies_Sackin}
Let $T$ be a rooted binary tree on $n$ leaves that has minimal Colless index, i.e. $\mathcal{C}(T) = c_n$. 
Then, $T$ has minimal Sackin index.
\end{proposition}

The proof of Proposition \ref{Colless_implies_Sackin} requires the following corollary from \citet{Fischer2018}:
\begin{corollary}[Corollary 4 in \citet{Fischer2018}] \label{cor_mareike}
Let $T$ be a rooted binary tree with $n \in \mathbb{N}_{\geq 2}$ leaves. 
Moreover, let $T = (T_a, T_b)$ be the standard decomposition of $T$ into its two maximal pending subtrees, let $n_i$ denote the number of leaves in $T_i$ for $i \in \{a, b\}$, respectively, such that $n_a \geq n_b$. Let $k_n \coloneqq \lceil \log_2 (n) \rceil$. 
Then, the following equivalence holds: 
$T$ has minimal Sackin index if and only if $T_a$ and $T_b$ have minimal Sackin index and $n_a - n_b \leq \min \{n-2^{k_n-1}, 2^{k_n}-n\}$.
\end{corollary}

\begin{proof}[Proof of Proposition \ref{Colless_implies_Sackin}]
We show the statement by induction on $n$.
For $n=1$, there is only one tree. This tree trivially has both minimal Colless index as well as minimal Sackin index, which completes the base case of the induction.

Suppose that the claim holds for all trees with fewer than $n$ leaves. 
Now, let $T=(T_a, T_b)$ be a tree with $n$ leaves and minimal Colless index, i.e. $\mathcal{C}(T)=c_n$. 
Let $n_a$ and $n_b$ denote the number of leaves of $T_a$ and $T_b$, respectively, and without loss of generality let $n_a \geq n_b$. 
Then from Theorem \ref{gfb_extreme} we have that:
\begin{align}
n_a - n_b \leq \begin{cases}
n-2^{k_n-1}, &\text{ if } n \in (2^{k_n-1}, 3 \cdot 2^{k_n-2}); \\
2^{k_n-2}, &\text{ if } n = 3 \cdot 2^{k_n-2}; \\
2^{k_n}-n, &\text{ if } n \in (3 \cdot 2^{k_n-2}, 2^{k_n}].
\end{cases}  \label{cases}
\end{align}

First, let $T$ have $n = 3 \cdot 2^{k_n-2}$ leaves. Then by \eqref{cases}, $n_a - n_b \leq 2^{k_n-2} = \min \{n-2^{k_n-1}, 2^{k_n}-n\} = \min \{2^{k_n-2}, 2^{k_n-2}\}$.
Moreover, by Lemma \ref{max_subtrees} both $T_a$ and $T_b$ have minimal Colless index and as $n_a, n_b < n$ they also have minimal Sackin index by the inductive hypothesis. In total, by Corollary \ref{cor_mareike} this implies that $T$ has minimal Sackin index.

Now, let $T$ have $n \in (2^{k_n-1}, 3 \cdot 2^{k_n-2})$ leaves. Then by \eqref{cases}, $n_a-n_b \leq n-2^{k_n-1} = \min \{n-2^{k_n-1}, 2^{k_n}-n\}$. Again, using Lemma \ref{max_subtrees} and the inductive hypothesis, by Corollary \ref{cor_mareike} this implies that $T$ has minimal Sackin index.

Last, let $T$ have $n \in (3 \cdot 2^{k_n-2}, 2^{k_n}]$ leaves. Then by \eqref{cases}, $n_a-n_b \leq 2^{k_n}-n = \min \{n-2^{k_n-1}, 2^{k_n}-n\}$, i.e. using Lemma \ref{max_subtrees} and the inductive hypothesis, by Corollary \ref{cor_mareike}, $T$ has minimal Sackin index. 
This completes the proof.
\end{proof}

As every tree with minimal Colless index has minimal Sackin index (while the converse is not true), a direct consequence of Proposition \ref{Colless_implies_Sackin} is the following corollary:

\begin{corollary} \label{Sackin_upperBound}
Let $\widetilde{c}(n)$ denote the number of binary rooted trees with $n$ leaves and minimal Colless index and let $\widetilde{s}(n)$ denote the number of binary rooted trees with $n$ leaves and minimal Sackin index. Then, we have $\widetilde{c}(n) \leq \widetilde{s}(n)$.
\end{corollary}

\begin{remark} \label{improved_bound}
Recall that not every tree with minimal Sackin index also has minimal Colless index. Thus, $\widetilde{s}(n)$ is not a sharp bound for $\widetilde{c}(n)$. While for a tree $T=(T_a, T_b)$ with $n_a \neq n_b$ and $n_a, n_b$ odd, it is totally possible to have minimal Sackin index (as long as $n_a$ and $n_b$ satisfy the conditions of Corollary \ref{cor_mareike}), for the Colless index this is never possible (cf. Theorem \ref{oddodd}). 
Thus, we can tighten the upper bound for $\widetilde{c}(n)$ by excluding all pairs of $n_a$ and $n_b$, where $n_a \neq n_b$ and $n_a, n_b$ odd from the recursion given in Theorem \ref{recursion}\footnote{In the last section of this manuscript, we will discuss how to improve this bound even further, based on the results by \citet{Coronado2019}, and even derive a recursive formula for the number of trees with minimal Colless index}. We denote the resulting tighter upper bound by $\widetilde{b}(n)$. In particular, we have for all $n \in \mathbb{N}$: $\widetilde{c}(n) \leq \widetilde{b}(n) \leq \widetilde{s}(n)$. This relation is depicted in Figure \ref{Fig_SackinColless}. 

Note that starting at $n=1$ and continuing up to $n=32$, the sequences $\widetilde{c}(n)$ and $\widetilde{b}(n)$ are $(1, 1, 1, 1, 1, 2, 1, 1, 1, 3, 3, 4, 3, 3, 1, 1, 1, 4, 6, 10, 16, 21, 13, 11, 13, 21, 16, 10, 6, 4, 1, 1)$ and $(1, 1, 1, 1, 1, 2, 1, 1, 1, 3, 3, 4, 3, 3, 1, 1, 1, 4, 6, 10, 16, 21, 25, 20, 25, 21, 16, 10, 6, 4, 1, 1)$, respectively (where $\widetilde{c}(n)$ results from an exhaustive enumeration of trees with minimal Colless index\footnote{See the discussion section for a recursive formula to calculate $\widetilde{c}(n)$ based on the results by \citet{Coronado2019}.} and $\widetilde{b}(n)$ is obtained from $\widetilde{s}(n)$ (cf. Theorem \ref{recursion}) by excluding all pairs of $n_a$ and $n_b$, where $n_a \neq n_b$ and $n_a, n_b$ odd). The sequence for $\widetilde{c}$ has been submitted to the On-Line Encyclopedia of Integer Sequences OEIS (\citet{OEIS}) as it so far had not been contained in it. It is currently under review there and will shortly be published as sequence \texttt{A307689}.
\end{remark}

\begin{figure}[h!]
	\centering
	\includegraphics[scale=0.3]{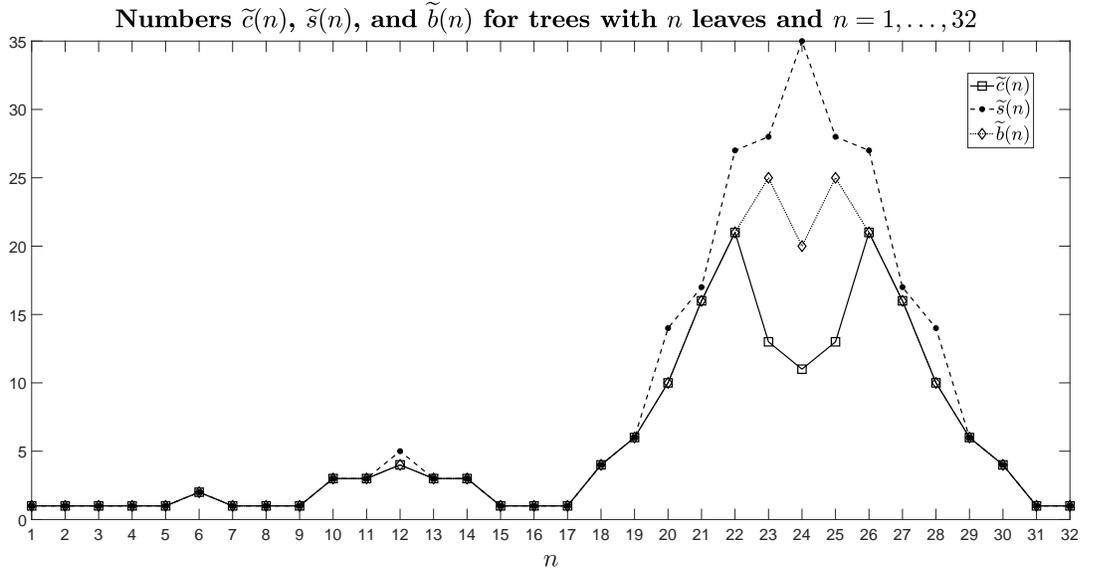}
	\caption{Number $\widetilde{c}(n)$ of rooted binary trees with $n$ leaves and minimal Colless index for $n=1, \ldots, 32$ (lines connecting these discrete data points are drawn for better readability).
	 $\widetilde{s}(n)$ and $\widetilde{b}(n)$ are two upper bounds for $\widetilde{c}(n)$, where $\widetilde{s}(n)$ is the number of rooted binary trees with minimal Sackin index (calculated according to Theorem \ref{recursion}) and $\widetilde{b}(n)$ is an improved upper bound (cf. Remark \ref{improved_bound}). Note that $\widetilde{c}(n)$ results from an exhaustive enumeration of all trees with minimal Colless index.}
	\label{Fig_SackinColless}
\end{figure}

To summarize, in this section we have further characterized trees with minimal Colless index. Additionally, we have given two upper bounds for the number of trees with minimal Colless index by first relating the Colless index to the Sackin index and then improving the obtained bound by using Theorem \ref{oddodd}.

\section{Discussion}
In this manuscript, we have thoroughly analyzed extremal properties of the Colless balance index. We have focused on the minimal Colless index of a tree with $n$ leaves and have both given a recursive formula as well as an explicit expression for this value, where the latter shows a surprising connection of the minimal Colless index to the \emph{Blancmange/Takagi} curve, a fractal curve. 
While the recursive formula\footnote{Note that this formula was independently also discovered by \citet{Coronado2019}.} directly yields a class of trees with minimal Colless index, namely the class of \emph{maximally balanced trees}, we have subsequently introduced another class of trees with minimal Colless index, namely the class of \emph{GFB trees}. Note that this class of trees might somehow be related to the explicit formula for the Colless value stated by \citet{Coronado2019}. On the other hand, our own explicit formula, as stated above, is more suitable to express the fractal structure of the minimal Colless index by relating it to the famous Blancmange curve.

Anyway, while the two mentioned classes of trees, i.e. maximally balanced trees and GFB trees, as well as their corresponding leaf partitionings yield trees with minimal Colless index, we have additionally shown that a tree $T=(T_a,T_b)$ with $n_a \neq n_b$ and $n_a, n_b$ odd, cannot have minimal Colless index, while for $n_a \neq n_b$ and $n_a, n_b$ even it may or may not have minimal Colless index. 

However, an independent full characterization of trees with minimal Colless index has recently been achieved by \citet{Coronado2019}, and the authors also characterize valid leaf partionings $n_a$ and $n_b$ for a tree with $n=n_a+n_b$ leaves and minimal Colless index. This characterization, which can be found in Proposition 11 of \citet{Coronado2019}, can be used to improve the bound $\widetilde{b}(n)$ as presented in Figure \ref{Fig_SackinColless} of our manuscript by summing only over those pairs $(n_a,n_b)$ that are valid due to this proposition. However, note that this improved bound $\widehat{b}(n)$ is still not sharp -- this can be seen e.g. by considering the tree depicted in Figure \ref{Fig_coronadobound}. This tree is a tree on $n=23$ leaves with minimal Sackin index and with $n_a=12$ and $n_b=11$, which is a combination of $n_a$ and $n_b$ that is explicitly allowed by Proposition 11 of \citet{Coronado2019} (which is correct, because there are in fact trees with 23 leaves and minimal Colless index and leaf partitioning $(n_a,n_b)=(12,11)$). However, as subtree $T_a$ with $n_a=12$ leaves consists of two maximal pending subtrees with 7 and 5 leaves, respectively (and thus a combination of two different odd numbers), by Theorem \ref{oddodd} in our manuscript, $T_a$ is not a tree with minimal Colless value, and thus by Lemma \ref{max_subtrees}, $T$ is not a tree with minimal Colless value, either. In fact, the minimal Colless value for $n=23$ is $c_{23}=10$ by Theorem \ref{colless_explicit}, but the tree depicted in Figure \ref{Fig_coronadobound} has Colless value $\mathcal{C}(T)=12$. 

\begin{figure}[h!]
	\centering
	\includegraphics[scale=0.3]{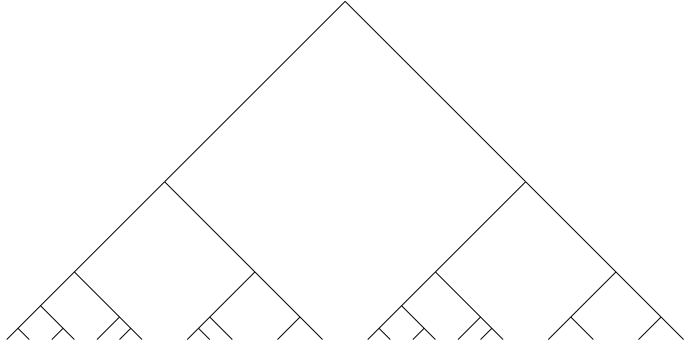}
	\caption{A rooted binary tree $T$ with $23$ leaves and minimal Sackin index (namely 28) and one possible Colless optimal leaf partitioning at the root, namely $n_a=12$ and $n_b=11$ (cf. \citet[Prop. 11]{Coronado2019}), but with Colless index 12, where the minimum would be 10. Thus, $T$ is not Colless minimal.}
	\label{Fig_coronadobound}
\end{figure}

By relating trees with minimal Colless index with trees with minimal Sackin index, we have shown that the two classic and most frequently used tree balance indices are actually closely related, and we have used this insight to present an upper bound for the number of Colless minimal trees. 

However, by denoting the set of valid $(n_a,n_b)$ pairs by $QB(n)$, i.e. $QB(n)=\{(n_a,n_b): n_a+n_b=n \mbox{ and $\exists$ a tree $T$ on $n$ leaves: $\mathcal{C}(T)=c_n$}$ $\mbox{ and with leaf partitioning $(n_a,n_b)$} \}$, a set which was characterized by \citet{Coronado2019}, one can quite easily derive a recursive formula for the number of Colless minima (in the same way as Theorem \ref{recursion} works for the number of Sackin minima):
\begin{itemize}
\item $\widetilde{c}(1)=1$
\item $\widetilde{c}(n)=  \sum\limits_{\substack{(n_a,n_b): \\n_a+n_b=n,\\ n_a\neq n_b, \\
(n_a,n_b)\in QB(n)}} \left( \widetilde{c}(n_a)\cdot \widetilde{c}(n_b) \right) +\widetilde{f}(n),$

\end{itemize} where $\widetilde{f}(n)=\begin{cases}0 & \mbox{if $n$ is odd} \\ {\widetilde{c}(\frac{n}{2})+1 \choose 2}& \mbox{else. }\end{cases}$. 

Note that the binomial coefficient in $\widetilde{f}$ prevents counting symmetries twice in the case where $n_a=n_b=\frac{n}{2}$. The correctness of this formula is a direct consequence of Lemma \ref{max_subtrees}, which implies that each Colless minimal tree has two maximal pending subtrees which are also Colless minimal, combined with the definition of the set $QB(n)$, which ensures that we only sum over pairs $(n_a,n_b)$ which indeed imply Colless minimal trees on $n$ leaves. 

Note that this is the first formula in the literature which enables us to calculate $\widetilde{c}(n)$, and we have submitted the resulting sequence to the Online Encyclopedia of Integer Sequences \citet{OEIS} as it was not previously listed there. It is currently under review there and will shortly be published as sequence \texttt{A307689}.
However, it would definitely be of interest to find an explicit formula for $\widetilde{c}(n)$ and to analyze if the fractal structure of the sequence of the minimal Colless index induced by the Blancmange curve is reflected in the sequence of the number of trees that achieve it (as is suggested by Figure \ref{Fig_SackinColless}). These are topics for future research. 

\section*{References}
\bibliographystyle{model1-num-names}\biboptions{authoryear}
\bibliography{references}

\section{Appendix}
\setcounter{theorem}{0}
\begin{theorem}
Let $c_n$ be the minimal Colless index for a rooted binary tree with $n$ leaves. Then, $c_1=c_2=0$, and for all $n \in \mathbb{N}_{\geq 1}$ we have
\begin{align*}
&c_{2n}=2c_n, \\
&c_{2n+1}=c_{n+1}+c_{n}+1.
\end{align*}
\end{theorem}
\begin{proof}
We show by induction on $n$ that
\begin{align*}
&c_{2n} \geq 2c_n \text{ and } \hspace*{70mm} &&\eqref{ib1} \\
&c_{2n+1} \geq c_{n+1}+c_{n}+1. \hspace*{70mm} &&\eqref{ib2} 
\end{align*}
Here we show the proof of \eqref{ib2}. The proof of \eqref{ib1} is given in the main part of the paper.

By Lemma \ref{colless_sum} and Lemma \ref{max_subtrees} we have that
\begin{align}
c_{2n+3}= \min\{&c_{2n+2}+c_1+2n+1, c_{2n+1}+c_2+2n-1, c_{2n}+c_3+2n-3, \dots, \nonumber \\
 &c_{n+3}+c_{n}+3, c_{n+2}+c_{n+1}+1 \}. \label{i2}
\end{align}

Let $n$ be even and consider \eqref{i2}, which can be rewritten by the inductive hypothesis as:
\begin{align}
c_{2n+3}\geq \min\{&2c_{n+1}+c_1+2n+1, c_{n+1}+c_n+1+2c_1+2n-1, 2c_{n}+c_{2}+c_1+1+2n-1, \dots, \nonumber \\
 &c_{\frac{n}{2}+1}+c_{\frac{n}{2}+2}+1+2c_{\frac{n}{2}}+3, c_{n+2}+c_{n+1}+1 \} \nonumber \\
 = \min\{&c_{n+1}+c_1+n+c_{n+1}+n+1, c_{n+1}+c_1+n+c_n+c_1+n-1+1, \nonumber \\
 &c_{n}+c_2+n-2+c_n+c_1+n-1+3,  \dots, \nonumber \\
 &c_{\frac{n}{2}+2}+c_{\frac{n}{2}}+2+c_{\frac{n}{2}+1}+c_{\frac{n}{2}}+1+1, c_{n+2}+c_{n+1}+1 \}. \label{i1.5}
\end{align}
Again by Lemma \ref{colless_sum} and Lemma \ref{max_subtrees}, we have that
\begin{align}
c_{n+2}= \min\{&c_{n+1}+c_1+n, c_{n}+c_2+n-2, c_{n-1}+c_3+n-4, \nonumber \\
 &\dots, 2c_{\frac{n}{2}+1} \}, \label{i1.6}
\end{align}
and thus we have for example $c_{n+2} \leq c_{n+1}+c_1+n$. Then by using \eqref{i1.6}, \eqref{i1.5} becomes the following 
\begin{align}
c_{2n+3}&\geq \min\{c_{n+2}+c_{n+1}+1+\underbrace{n}_{\geq 0}, c_{n+2}+c_{n+1}+1, c_{n+2}+c_{n+1}+1+\underbrace{2}_{\geq 0}, \nonumber \\ 
&\qquad \dots, c_{n+2}+c_{n+1}+1, c_{n+2}+c_{n+1}+1 \} \nonumber \\
&= c_{n+2}+c_{n+1}+1.
\end{align}
This completes the proof of Equation \eqref{ib2} for $n$ even.

Now, let $n$ be odd and consider \eqref{i2}, which can be rewritten similar to the case before by the inductive hypothesis:
\begin{align}
c_{2n+3}\geq \min\{&2c_{n+1}+c_1+2n+1, c_{n+1}+c_n+1+2c_1+2n-1, 2c_{n}+c_{2}+c_1+1+2n-1, \nonumber \\
 &\dots, 2c_{\frac{n+1}{2}+1}+c_{\frac{n+1}{2}}+c_{\frac{n-1}{2}}+1+3, c_{n+2}+c_{n+1}+1 \} \nonumber \\
 = \min\{&c_{n+1}+c_1+n+c_{n+1}+n+1, c_{n+1}+c_1+n+c_n+c_1+n-1+1, \nonumber \\
 &c_{n}+c_2+n-2+c_n+c_1+n-1+3, \dots,  \nonumber \\
 &c_{\frac{n+1}{2}+1}+c_{\frac{n+1}{2}}+1+c_{\frac{n+1}{2}+1}+c_{\frac{n-1}{2}}+2+1, c_{n+2}+c_{n+1}+1 \}. \label{i1.7}
\end{align}
By Lemma \ref{colless_sum} and Lemma \ref{max_subtrees}, we have that
\begin{align}
c_{n+2}= \min\{&c_{n+1}+c_1+n, c_{n}+c_2+n-2, c_{n-1}+c_3+n-4, \nonumber \\
 &\dots, c_{\frac{n+1}{2}+1}+c_{\frac{n+1}{2}}+1 \}, \label{i1.8}
\end{align}
and thus we have for example $c_{n+2} \leq c_{n+1}+c_1+n$. Then by using \eqref{i1.8}, \eqref{i1.7} becomes the following 
\begin{align}
c_{2n+3}&\geq \min\{c_{n+2}+c_{n+1}+1+\underbrace{n}_{\geq 0}, c_{n+2}+c_{n+1}+1, c_{n+2}+c_{n+1}+1+\underbrace{2}_{\geq 0}, \nonumber \\ 
&\qquad\dots, c_{n+2}+c_{n+1}+1, c_{n+2}+c_{n+1}+1 \} \nonumber \\
&= c_{n+2}+c_{n+1}+1.
\end{align}
This completes the proof of Equation \eqref{ib2} for $n$ odd.
Thus, also \eqref{ib2} holds for all $n$.
\end{proof}

\setcounter{lemma}{4}
\begin{lemma}
Let $n \in (2^{k_n-1},2^{k_n})$ be odd, where $k_n = \lceil \log_2 (n) \rceil$, and let $n-1>2^{k_n-1}$. Moreover, let $s(x) = \min\limits_{z \in \mathbb{Z}} \vert x-z \vert$ and let 
$$ f_i(n) \coloneqq \frac{s(2^{i-k_n+1} \cdot n)}{2^{i-k_n+1}}.$$
Then for $0 \leq i \leq k_n-3$,
$$f_i(n+1) + f_i(n-1)=2 \cdot f_i(n).$$
\end{lemma}
\begin{proof}
Let $n \in (2^{k_n-1},2^{k_n})$ be odd, let $n-1>2^{k_n-1}$ and let $0 \leq i \leq k_n-3$.
Then,
\begin{align*}
f_i(n) &=\frac{s(2^{i-k_n+1} \cdot n)}{2^{i-k_n+1}} \\
&= \frac{\min\limits_{z \in \mathbb{Z}}\vert 2^{i-k_n+1} \cdot n -z \vert}{2^{i-k_n+1}} \\
&= \min\limits_{z \in \mathbb{Z}} \bigg\vert \frac{2^{i-k_n+1} \cdot n -z}{2^{i-k_n+1}} \bigg\vert \\
&= \min\limits_{z \in \mathbb{Z}} \bigg\vert n - \frac{z}{2^{i-k_n+1}} \bigg\vert, 
\end{align*}
i.e. $f_i(n)$ is the minimal distance of $n$ to an integer multiple of $2^{k_n -i-1}$. 

If $f_i(n) = n - \frac{z}{2^{i-k_n+1}}$, then $n \in [\frac{z}{2^{i-k_n+1}}, \frac{z+1}{2^{i-k_n+1}})$.

If $f_i(n) = \frac{\tilde{z}}{2^{i-k_n+1}} -n$, then $n \in (\frac{\tilde{z}-1}{2^{i-k_n+1}}, \frac{\tilde{z}}{2^{i-k_n+1}}]$. Let $z \coloneqq \tilde{z} -1$, then $n \in (\frac{z}{2^{i-k_n+1}}, \frac{z+1}{2^{i-k_n+1}}]$.

Note that for $0 \leq i \leq k_n-3$ we have $k_n-i-1 \geq 2$. Thus, $2^{k_n-i-1} \in \mathbb{N}$ and $2^{k_n-i-1}$ is a power of $2$. Therefore, $\frac{z}{2^{i-k_n+1}} = 2^{k_n-i-1} \cdot z$ is an even number for all $z \in \mathbb{Z}$, but $n$ is odd by assumption.

Therefore, in both cases we have that $n \in (\frac{z}{2^{i-k_n+1}}, \frac{z+1}{2^{i-k_n+1}})$ for some $z \in \mathbb{Z}$.

Let $m$ be the middle of the interval $(\frac{z}{2^{i-k_n+1}}, \frac{z+1}{2^{i-k_n+1}})$, i.e. 
\begin{align*}
m &= \frac{z}{2^{i-k_n+1}} + \frac{1}{2} \cdot \bigg(\frac{z+1}{2^{i-k_n+1}} - \frac{z}{2^{i-k_n+1}}\bigg) \\
&= \frac{z}{2^{i-k_n+1}} + \frac{1}{2} \cdot \frac{z+1-z}{2^{i-k_n+1}} \\
&= \frac{z}{2^{i-k_n+1}} + \frac{1}{2^{i-k_n+2}}.
\end{align*}

For $0 \leq i \leq k_n-3$ we have that $\frac{1}{2^{i-k_n+2}} = 2^{k_n-i-2} \geq 2^{k_n-(k_n-3)-2} =2$. Thus, $\frac{1}{2^{i-k_n+2}} \in \mathbb{N}$ and $\frac{1}{2^{i-k_n+2}}$ is a power of 2, which leads to the fact that $\frac{1}{2^{i-k_n+2}}$ is even.

As we have already seen, $\frac{z}{2^{i-k_n+1}}$ is even for all $z \in \mathbb{Z}$ as well.

Therefore, $m= \frac{z}{2^{i-k_n+1}} + \frac{1}{2^{i-k_n+2}}$ is an even number and the fact that $n$ is odd gives us $n \neq m$. 

We now distinguish two cases: $n > m$ and $n < m$.
\begin{enumerate}
\item If $n > m$, then we have that $n-1 \geq m$ and thus $n-1, n, n+1 \in [m,\frac{z+1}{2^{i-k_n+1}}]$.
Therefore, we have 
$f_i(n) = \frac{z+1}{2^{i-k_n+1}} -n$, $f_i(n+1) = \frac{z+1}{2^{i-k_n+1}} -(n+1)$ and $f_i(n-1) = \frac{z+1}{2^{i-k_n+1}} -(n-1)$, which gives us
\begin{align*}
f_i(n-1) + f_i(n+1) &= \frac{z+1}{2^{i-k_n+1}} -(n-1) + \frac{z+1}{2^{i-k_n+1}} -(n+1) \\
&= 2 \cdot \bigg( \frac{z+1}{2^{i-k_n+1}} -n \bigg) \\
&= 2 \cdot f_i(n).
\end{align*}
\item If $n < m$, then we have that $n+1 \leq m$ and thus $n-1, n, n+1 \in [\frac{z}{2^{i-k_n+1}},m]$. Therefore, we have $f_i(n) = n - \frac{z}{2^{i-k_n+1}}$, $f_i(n+1) = n+1 - \frac{z}{2^{i-k_n+1}}$ and $f_i(n-1) = n - 1 - \frac{z}{2^{i-k_n+1}}$, which gives us
\begin{align*}
f_i(n-1) + f_i(n+1) &= n - 1 - \frac{z}{2^{i-k_n+1}} + n+1 - \frac{z}{2^{i-k_n+1}} \\
&= 2 \cdot \bigg(n-\frac{z}{2^{i-k_n+1}} \bigg) \\
&= 2 \cdot f_i(n).
\end{align*}
\end{enumerate}
So in both cases the claim holds, which completes the proof.
\end{proof}

\setcounter{proposition}{0}
\begin{proposition} 
Let $n \in \mathbb{N}_{\geq 1}$ and let $k_n \coloneqq \lceil \log_2(n) \rceil$. Then, we have for the minimal Colless index $c_n$:
\begin{enumerate}
\item If $n=2^{k_n}+1$, then $c_n=k_n$.
\item If $n=2^{k_n}-1$, then $c_n=k_n-1$.
\item For $n \in (2^{k_n-1},2^{k_n})$ and $j \in \{1,\dots, 2^{k_n-1}-1 \}$ we have $c_{2^{k_n-1}+j} = c_{2^{k_n}-j}.$
\end{enumerate}
\end{proposition}

\begin{proof}
\begin{enumerate}
\item
The proof is by induction on $k_n$. For $k_n=0$ we have $c_{2^{0}+1}=c_2=0=k_n$, which gives the base case of the induction. Now, we assume that the claim holds up to $k_n$, and we show that it also holds for $k_n+1$. Let $n=2^{k_n+1}+1$. Then,
\begin{align*}
c_{2^{k_n+1}+1} &= c_{2^{k_n}+1} + c_{2^{k_n}} +1 &&\text{by Theorem } \ref{colless_minimum} \\
&= c_{2^{k_n}+1} +1 &&\text{by Theorem } \ref{min_colless} \\
&= k_n+1, &&\text{by the inductive hypothesis}
\end{align*}
which completes the proof.
\item 
The proof is by induction on $k_n$. For $k_n=1$ we have $c_{2^{1}-1}=c_1=0=k_n-1$, which gives the base case of the induction. Now, we assume that the claim holds up to $k_n$, and we show that it also holds for $k_n+1$. Let $n=2^{k_n+1}-1$. Then,
\begin{align*}
c_{2^{k_n+1}-1} &= c_{2^{k_n}} + c_{2^{k_n}-1} +1 &&\text{by Theorem } \ref{colless_minimum} \\
&= c_{2^{k_n}-1} +1 &&\text{by Theorem } \ref{min_colless} \\
&= k_n-1+1 &&\text{by the inductive hypothesis} \\
&= k_n,
\end{align*}
which completes the proof.
\item 
Let $n \in (2^{k_n-1},2^{k_n})$ and let $j \in \{1,\dots, 2^{k_n-1}-1 \}$. Then,
\begin{align*}
c_{2^{k_n-1}+j} &= \sum\limits_{i=0}^{k_n-2} \frac{s(2^{i-k_n+1} \cdot (2^{k_n-1}+j))}{2^{i-k_n+1}} \, \text{ by Theorem \ref{colless_explicit}} \\
&= \sum\limits_{i=0}^{k_n-2} \frac{s(2^i + 2^{i-k_n+1} \cdot j)}{2^{i-k_n+1}} \\
&= \sum\limits_{i=0}^{k_n-2} \frac{s(2^{i-k_n+1} \cdot j)}{2^{i-k_n+1}} \; \text{ by Lemma \ref{subbaditivity_of_s}, Part 2,  as } 2^{i} \in \mathbb{Z} \\
&= \sum\limits_{i=0}^{k_n-2} \frac{s(-2^{i-k_n+1} \cdot j)}{2^{i-k_n+1}} \\
&= \sum\limits_{i=0}^{k_n-2} \frac{s(2^{i+1}-2^{i-k_n+1} \cdot j)}{2^{i-k_n+1}} 
\; \text{ by Lemma \ref{subbaditivity_of_s}, Part 2,  as } 2^{i+1} \in \mathbb{Z}\\
&= \sum\limits_{i=0}^{k_n-2} \frac{s(2^{i-k_n+1} \cdot (2^{k_n}-j))}{2^{i-k_n+1}} \\ 
&= c_{2^{k_n}-j} \, \text{ by Theorem \ref{colless_explicit}}.
\end{align*}
\end{enumerate}
\end{proof}

\setcounter{proposition}{1}
\begin{proposition}
Let $T_n^{gfb}$ be a GFB tree with $n\geq 2$ leaves and standard decomposition $T_n^{gfb}=(T_a,T_b)$. Let $n_a$ and $n_b$ denote the number of leaves of $T_a$ and $T_b$, respectively, such that without loss of generality, $n_a \geq n_b$. Let $k_n \coloneqq \lceil \log_2 (n) \rceil$, i.e. $n \in (2^{k_n-1}, 2^{k_n}]$. Then, we have:
	\begin{enumerate}
	\item If $n \in (2^{k_n-1}, 3 \cdot 2^{k_n-2})$, we have $n_a = n-2^{k_n-2}$ and $n_b = 2^{k_n-2}$. In particular, $T_b$ is the fully balanced tree of height $k_n-2$ and we have $k_{n_a} \coloneqq \lceil \log_2 (n_a) \rceil=k_n-1$. 
	\item If $n = 3 \cdot 2^{k_n-2}$, we have $n_a = 2^{k_n-1}$ and $n_b = 2^{k_n-2}$. In particular, $T_a$ is the fully balanced tree of height $k_n-1$ and $T_b$ is the fully balanced tree of height $k_n-2$.
	\item If $n \in (3 \cdot 2^{k_n-2}, 2^{k_n}]$, we have $n_a = 2^{k_n-1}$ and $n_b=n-2^{k_n-1}$. In particular, $T_a$ is the fully balanced tree of height $k_n-1$ and we have $k_{n_b} \coloneqq \lceil \log_2 (n_b) \rceil= k_n-1$.
	\end{enumerate}
\end{proposition}

The proof of Proposition \ref{GFB_Decomposition} requires the following lemma. 
\setcounter{lemma}{8}
\begin{lemma} \label{subtree_in_common}
Let $n \in \mathbb{N}_{\geq 3}$ and $n$ odd. Then, 
	\begin{itemize}
	\item $T_n^{gfb}$ and $T_{n-1}^{gfb}$ have a common maximal pending subtree and 
	\item $T_n^{gfb}$ and $T_{n+1}^{gfb}$ have a common maximal pending subtree.
	\end{itemize}	 
\end{lemma}

\begin{proof}
Let $n \in \mathbb{N}_{\geq 3}$ and $n$ odd. 
As $n$ is odd, the first $\frac{n-1}{2}$ iterations of the while-loop in Algorithm \ref{alg_gfb} result in $\frac{n-1}{2}$ trees of size 2 and one tree of size 1, which in the $\frac{n+1}{2}^{\text{th}}$ iteration is clustered with a tree of size 2 to form a tree of size 3. Note that as the algorithm continues clustering trees, in each iteration there will be precisely one tree $T_i^{odd}$ with an odd number $s(i)$ of leaves, while all others have an even number of leaves. However, note that this unique tree with $s(i)$ leaves is treated by the algorithm like a tree with $s(i)-1$ leaves, except that it is clustered as late as possible, i.e. when all other elements in $treeset$ with $s(i)-1$ leaves (if there are any) have already been clustered. On the other hand, however, this tree is treated by the algorithm like a tree with $s(i)+1$ leaves,  except that it is clustered as early as possible, i.e. before any other elements in $treeset$ with $s(i)+1$ leaves (if there are any) get clustered.

To summarize, after the first $\frac{n+1}{2}$ iterations of the while-loop, $treeset$ contains a unique tree $T_i^{odd}$ with an odd number $s(i)$ of leaves, which at the same time
\begin{enumerate}
\item [(i)] is treated like a tree with $s(i)-1$ leaves, but is clustered as late as possible;
\item [(ii)] is treated like a tree with $s(i)+1$ leaves, but is clustered as soon as possible.
\end{enumerate}

Now, first consider Algorithm \ref{gfb} for $n-1$, which is an even number.
After the first $\frac{n-3}{2}$ iterations of the while-loop, $treeset$ contains $\frac{n-3}{2}$ trees of size 2 and two trees of size 1, which are clustered last to form the last cherry. We keep tracking one leaf $u$ of this cherry throughout the algorithm. The algorithm at this stage contains only cherries, which are all isomorphic, so without loss of generality, we may assume that $u$ is contained in the one that gets clustered with another tree last, i.e. after all other cherries have been clustered. We continue like this, always assuming without loss of generality (when there is more than one tree in $treeset$ of the same size as the tree that contains $u$) that $u$ is in the last one to be clustered. By (i), this means that if we replace $u$ in $T_{n-1}^{gfb}$ by a cherry, we derive $T_{n}^{gfb}$. This is due to the fact that in the analogous step where $treeset$ for $n-1$ only contains cherries, $treeset$ for $n$ will contain only cherries and a tree containing three leaves. This triplet will subsequently act like a cherry, but like the one that happens to be clustered last. So we identify the cherry of the triplet with leaf $u$ of this last cherry to see the correspondence between $T_{n-1}^{gfb}$ and $T_{n}^{gfb}$. Note that this also directly implies that $T_{n-1}^{gfb}$ and $T_{n}^{gfb}$ share a common maximal pending subtree -- namely the one that does {\em not} contain $u$.

Note that by (ii), an analogous procedure for $n+1$ leads to $T_{n+1}^{gfb}$ and $T_{n}^{gfb}$ sharing a common maximal pending subtree. In this case, we track a cherry in $T_{n+1}^{gfb}$, namely the one that happens to be clustered first, and replace it by a single leaf to see the correspondence between $T_{n+1}^{gfb}$ and $T_{n}^{gfb}$. 

So $T_{n}^{gfb}$ shares a common maximal pending subtree with both $T_{n-1}^{gfb}$ and $T_{n+1}^{gfb}$, respectively. This completes the proof.
\end{proof}

Note that the main idea of above proof is illustrated in Figure \ref{Fig_treeset}.

\begin{sidewaysfigure}
	\centering
	\includegraphics[scale=0.1]{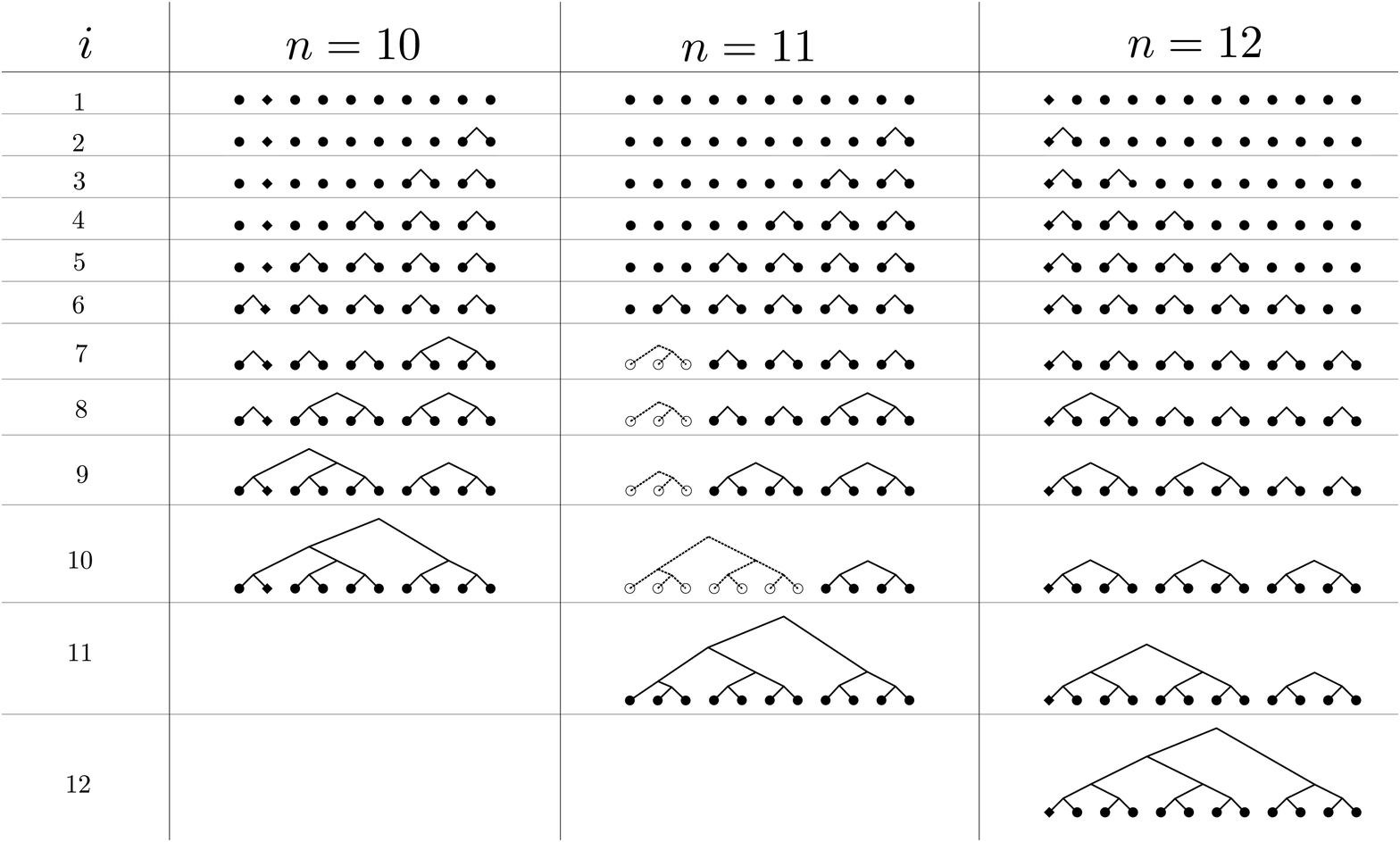}
	\caption{Content of $treeset$ before the $i^{\rm th}$ iteration of the while-loop in Algorithm \ref{alg_gfb} for $n=10, n=11$ and $n=12$. In case of $n=11$, the tree depicted in dashed lines for $i=7, \ldots, 10$, depicts the \emph{unique} tree in $treeset$ with an odd number of leaves. For $n=10$, the leaf depicted as a diamond represents leaf $u$ used in the proof of Lemma \ref{subtree_in_common}. Note that the tree containing this leaf is always clustered as late as possible. In case of $n=12$, the leaf depicted as a diamond again represents leaf $u$ used in the proof of Lemma \ref{subtree_in_common}. In this case, the tree containing this leaf is always clustered as soon as possible. The last tree depicted in each column represents the GFB tree. Note that $T_n^{gfb}$ can be obtained from $T_{n-1}^{gfb}$ by replacing the leaf depicted as a diamond by a cherry. Moreover, $T_n^{gfb}$ can be obtained from $T_{n+1}^{gfb}$ by replacing the cherry containing the diamond leaf by a single leaf. }
	\label{Fig_treeset}
\end{sidewaysfigure}

\begin{proof}[Proof of Proposition \ref{GFB_Decomposition}]
Our proof strategy is as follows: In order to simplify the proof, instead of analyzing all three cases separately, we investigate only two cases, which correspond to the first and the last case but -- by adding the respective interval bound -- directly imply the second case. In particular, we inductively prove the following statements: 
\begin{enumerate} [(i)]	
\item  \label{inew} If $n \in (2^{k_n-1}, 3 \cdot 2^{k_n-2}]$, we have $n_a = n-2^{k_n-2}$ and $n_b = 2^{k_n-2}$. In particular, $T_b$ is the fully balanced tree of height $k_n-2$ and we have $k_{n_a} \coloneqq \lceil \log_2 (n_a) \rceil=k_n-1$. 
\item   \label{iinew} If $n \in [3 \cdot 2^{k_n-2}, 2^{k_n}]$, we have $n_a = 2^{k_n-1}$ and $n_b=n-2^{k_n-1}$. In particular, $T_a$ is the fully balanced tree of height $k_n-1$ and we have $k_{n_b} \coloneqq \lceil \log_2 (n_b) \rceil \geq k_n-2$, where equality holds precisely if $n=3 \cdot 2^{k_n-2}$.
\end{enumerate}

$n=2$ is the base case for Case \eqref{iinew}. In this case, we have $k_n=1$, as we have $n=2^1=2^{k_n}$. Applying Algorithm \ref{gfb} to 2 leaves results in a cherry. Thus, $n_a=1=2^{1-1}=2^{k_n-1}$ and $n_b = 1 = 2-2^{1-1}=n-2^{k_n-1}$. In particular, $T_a$ is a fully balanced trees of height $k_n-1=0$. Moreover, as $n_b=1$, $T_b$ is also a fully balanced tree of height $k_n-1=0$, and thus in particular $k_{n_b}=k_n-1$. \\
$n=3$ is the base case for Case \eqref{inew} (note that it is at the same time an example of Case \eqref{iinew} as we have $n=3 \cdot 2^{2-2}=3 \cdot 2^{k_n-2}$). In this case, Algorithm \ref{gfb} returns a so-called triplet, i.e. a tree $T=(T_a,T_b)$, where $T_a$ consists of two leaves forming a cherry and $T_b$ consists only of one leaf. Thus, $n_a=2=2^{2-1}=2^{k_n-1}$ and $n_b = 1 = 3-2^{2-1}=n-2^{k_n-1}$. In particular, $T_a$ is a fully balanced tree of height $k_n-1=1$. Moreover, as $n_b=1$, $T_b$ is a fully balanced tree of height $k_n-2=0$, and thus in particular $k_{n_b}=k_n-2$. (Note that this case also shows how Cases \eqref{inew} and \eqref{iinew} together imply statement 2. of the proposition.)\\

Now, let $n \geq 4$ and assume that \eqref{inew} and \eqref{iinew} hold for up to $n-1$ leaves. We now consider $n$ leaves, where we distinguish two cases:
\begin{itemize}
\item $n$ is an even number: \\
If $n$ is even, Algorithm \ref{alg_gfb} results in a tree $T^{gfb}_n$ with $\frac{n}{2}$ cherries (because in each of the first $\frac{n}{2}$ iterations of the while-loop two trees of size 1 are merged into a cherry). We now consider the tree $T'$ with $n'=\frac{n}{2}$ leaves that is obtained from $T^{gfb}_n$ by replacing all cherries with single leaves. 
Note that as $T^{gfb}_n$ is a GFB tree, so is $T'$ (because as soon as Algorithm \ref{alg_gfb} only has cherries to choose from, they are treated like leaves). Moreover, as $n'<n$, we can use the inductive hypothesis to infer the sizes $n_a'$ and $n_b'$ of the two maximal pending subtrees of $T'$. Exemplarily, consider Case \eqref{inew}, i.e. $n \in (2^{k_n-1}, 3 \cdot 2^{k_n-2}]$, i.e. $n' \in (2^{k_n-2}, 3 \cdot 2^{k_n-3}]$: \\
By the inductive hypothesis, $n_a' = n'-2^{k_n-3} = \frac{n}{2} - 2^{k_n-3}$ and $n_b' = 2^{k_n-3}$. In particular, $T_b'$ is the fully balanced tree of height $k_n-3$ and we have for $T_a'$: $k_{n_a'} = \lceil \log_2(n_a') \rceil = k_n-2$. We now go back from $T'$ to $T^{gfb}_n = (T_a, T_b)$ by replacing all leaves of $T'$ with cherries. This implies
\begin{align*}
n_a &= 2 \cdot n_a' = 2 \cdot \left( \frac{n}{2} - 2^{k_n-3} \right) = n - 2^{k_n-2}, \\
n_b &= 2 \cdot n_b' = 2 \cdot 2^{k_n-3} = 2^{k_n-2}.
\end{align*}
In particular, $T_b$ is the fully balanced tree of height $k_n-2$ (because replacing all leaves of a fully balanced tree of height $k_n-3$ with cherries results in a fully balanced tree of height $k_n-2$). Moreover, as $n_a = 2 \cdot n'_a$ and $k_{n'_a} = k_n-2$, we can conclude that
		\begin{align*}
		k_{n_a} &= \lceil \log_2(n_a) \rceil
		= \lceil \log_2(2 \cdot n'_a) \rceil
		= 1 + \lceil \log_2(n'_a) \rceil
		= k_n - 1.
		\end{align*}

This completes the proof for the case that $n$ is even and contained in $ (2^{k_n-1}, 3 \cdot 2^{k_n-2}]$. The case where $n$ is even and contained in $ [3 \cdot 2^{k_n-2}, 2^{k_n}]$ follows analogously: Here, we derive $n_a=2^{k_n-1} $and thus $n_b=n-2^{k_n-1}$. Therefore, $T_a$ is the fully balanced tree of height $k_n-1$. However, for $T_b$, we have to consider the case $n=3 \cdot 2^{k_n-2}$ separately. If $n>3 \cdot 2^{k_n-2}$, we have $n_b=n-2^{k_n-1}>2^{k_n-2}$. Therefore, $k_{n_b}=k_n-1$ (note that $n_b \leq n_a$ and thus $k_{n_b}\leq k_{n_a}=k_n-1$). However, if $n=3 \cdot 2^{k_n-2}$, we have $n_b=2^{k_n-2}$ and thus $k_{n_b}= k_n-2$. This completes the proof for even values of $n$. 

\item $n$ is an odd number: \\
If $n$ is odd, $n-1$ and $n+1$ are even, and as $n-1< n$ and $\frac{n+1}{2} < n$ as $n \geq 2$ by assumption, we can use an inductive argument to infer the leaf partitioning for $T^1 \coloneqq T_{n-1}^{gfb}$ and $T^2\coloneqq T_{n+1}^{gfb}$ (for $T^2$, we will use the fact that that $T^2=T_{n+1}^{gfb}$ contains $\frac{n+1}{2}$ cherries, apply the inductive hypothesis to a tree $\tilde{T}^2$ obtained from $T^2$ by replacing all cherries with single leaves and go back to $T^2$). In the following, let $T^1=(T^1_a, T^1_b)$ and $T^2=(T^2_a, T^2_b)$ denote the standard decompositions of $T^1$ and $T^2$, respectively.
Moreover, as $n$ is odd, in particular, as $n \geq 2$, $n \neq 2^{k_n-1}$, $n\neq 3 \cdot 2^{k_n-2}$ and $n \neq 2^{k_n}$ for all $k_n \in \mathbb{N}$. This implies that $n-1$, $n$ and $n+1$ are all contained together in the same interval, i.e. either all of them are in $[2^{k_n-1}, 3 \cdot 2^{k_n-2}]$ or all of them are in $[3 \cdot 2^{k_n-1}, 2^{k_n}]$.
We now distinguish between these two cases:
\begin{enumerate}
	\item $ n-1$, $n$, $n+1 \in [2^{k_n-1}, 3 \cdot 2^{k_n-2}]$:
		\begin{itemize}
		\item 
	First, suppose that $n-1$, $n$ and $ n+1 \in (2^{k_n-1}, 3 \cdot 2^{k_n-2}]$, i.e. that $n-1 > 2^{k_n-1}$.
		As $n-1 < n$, by the inductive hypothesis, we have for $T^1$:
		\begin{align*}
		n_a^1 = n  - 1 - 2^{k_{n-1}-2} \text{ and } n_b^1 = 2^{k_{n-1}-2},
		\end{align*}
		where $k_{n-1} \coloneqq \lceil \log_2 (n-1)\rceil$. Note that $k_{n-1} = k_n$ as $n-1 > 2^{k_n-1}$. This implies: 
		\begin{align*}
		n_a^1 = n  - 1 - 2^{k_{n}-2} \text{ and } n_b^1 = 2^{k_{n}-2},
		\end{align*}
and thus, in particular, $T_b^1$ is the fully balanced tree of height $k_{n}-2$ by the inductive hypothesis.
		Moreover, for $T^2$, we can use the same inductive argument as in the case $n$ even (i.e. we use the fact that $T^2=T_{n+1}^{gfb}$ contains $\frac{n+1}{2}$ cherries, apply the inductive hypothesis to a tree $\tilde{T}^2$ obtained from $T^2$ by replacing all cherries with single leaves and go back to $T^2$) to conclude that:
		\begin{align*}
		n_a^2 = n + 1 - 2^{k_n-2} \text{ and } n_b^2 = 2^{k_n-2},
		\end{align*}
		where $T_b^2$ is the fully balanced tree of height $k_n-2$ by the inductive hypothesis.
		
	Now, by Lemma \ref{subtree_in_common}, $T_n^{gfb}$ shares a common subtree with $T^1$, but it has one more leaf than $T^1$, so we can conclude that one of the following two cases must hold:
	\begin{align}
	n_a &= n_a^1 + 1 = n - 2^{k_n-2} \text{ and } n_b = n_b^1 = 2^{k_n-2} \label{var11} \text{ or } \\
	n_a &= n_a^1 = n - 1 - 2^{k_n-2} \text{ and } n_b = n_b^1 + 1 = 2^{k_n-2} + 1. \label{var12}
	\end{align}
	On the other hand, as $T_n^{gfb}$ by Lemma \ref{subtree_in_common} also shares a common subtree with $T^2$, but has one leaf less, we can conclude that one of the following two cases must hold:
	\begin{align}
	n_a &= n_a^2 - 1 = n - 2^{k_n-2} \text{ and } n_b = n_b^2 = 2^{k_n-2} \label{var21} \text{ or } \\
	n_a &= n_a^2 = n + 1 - 2^{k_n-2} \text { and } n_b = n_b^2 - 1 = 2^{k_n-2}-1. \label{var22}
	\end{align}
	As both one of Eq. \eqref{var11} and \eqref{var12} as well as one of Eq. \eqref{var21} and \eqref{var22} have to hold, we can conclude that Eq. \eqref{var11} = Eq. \eqref{var21} holds, as all other combinations are mutually exclusive. In particular, the subtree $T_b$ of $T_n^{gfb}$ is a fully balanced tree of height $k_n-2$ (it is the maximal pending subtree that $T_n^{gfb}$ shares with both $T^1$ and $T^2$). Moreover, as $n \in (2^{k_n-1}, 3 \cdot 2^{k_n-2})$, we have that $n_a = n - 2^{k_n-2} \in (2^{k_n-2}, 2^{k_n-1})$. In particular, $k_{n_a} = \lceil \log_2(n_a) \rceil = k_n-1$.
	
	\item Now, if $n-1 = 2^{k_n-1}$, by the inductive hypothesis, we have for $T^1 = T_{n-1}^{gfb}$:
		\begin{align*}
		n_a^1 &= 2^{k_n-2}, \\ 
		n_b^1 &= (n-1) - 2^{k_n-2} = 2^{k_n-1} - 2^{k_n-2} = 2^{k_n-2},
		\end{align*}
		where both $T_a^1$ and $T_a^2$ are fully balanced trees of height $k_n-2$ by the inductive hypothesis.
		
		For $T^2 = T_{n+1}^{gfb}$, we use the same inductive argument based on $\frac{n+1}{2}$ as above to conclude that:
		\begin{align*}
		n_a^2 &= (n+1) - 2^{k_n-2} = 2^{k_n-1} + 2 - 2^{k_n-2} = 2^{k_n-2} + 2, \\
		n_b^2 &= 2^{k_n-2}.
		\end{align*}
		In particular, $T_b^2$ is a fully balanced tree of height $k_n-2$ by the inductive hypothesis.
		
	Now, just as in the previous case, we exploit Lemma \ref{subtree_in_common} to conclude that $n_a = n_a^1 + 1 = n_a^2-1 = 2^{k_n-2} + 1$ and $n_b = n_b^2 = 2^{k_n-2} $. 

	In particular, subtree $T_b$ of $T_n^{gfb}$ is a fully balanced tree of height $k_n-2$ (it is the maximal pending subtree $T_n^{gfb}$ shares with both $T^1$ and $T^2$). Moreover, as $n \in (2^{k_n-1}, 3 \cdot 2^{k_n-2})$, we have that $n_a = n - 2^{k_n-2} \in (2^{k_n-2}, 2^{k_n-1})$. In particular, $k_{n_a} = \lceil \log_2(n_a) \rceil = k_n-1$. 
\end{itemize}	

	\item $n-1$, $n$ and $n+1 \in [3 \cdot 2^{k_n-2}, 2^{k_n}]$:
	By the same inductive argument as above, we can derive the following leaf partitioning for tree $T^1=T_{n-1}^{gfb}$ and tree $T^2=T_{n+1}^{gfb}$:
	\begin{align*}
	n_a^1 = 2^{k_n-1} \text{ and } n_b^1 = n - 1 - 2^{k_n-1}, \\
	n_a^2 = 2^{k_n-1} \text{ and } n_b^2 = n + 1 - 2^{k_n-1},
	\end{align*}
	where $T_a^1=T_a^2$ is the fully balanced tree of height $k_n-1$ by the inductive hypothesis and we have  $k_{n_b^1} = k_{n_b^2} = k_n-1$.
	Now, we again exploit Lemma \ref{subtree_in_common}, i.e. the fact that $T_n^{gfb}$ shares a common subtree with $T^1$ and a common subtree with $T^2$, to conclude that  the subtree $T_a$ of $T_n^{gfb}$ equals $T_a^1 = T_a^2$, and thus, as it has $ 2^{k_n-1} $ leaves, it is a fully balanced tree of height $k_n-1$. Moreover, as $n \in (3 \cdot 2^{k_n-2}, 2^{k_n})$, we have that $n_b = n - 2^{k_n-1} \in (2^{k_n-2}, 2^{k_n-1})$. In particular, $k_{n_a} = \lceil \log_2(n_a) \rceil = k_n-1$. This completes the proof.
\end{enumerate}
\end{itemize}
\end{proof}

\setcounter{lemma}{6}
\begin{lemma}
Let $s(x) = \min\limits_{z \in \mathbb{Z}}\vert  x-z \vert$, i.e. $s(x)$ is the distance from $x$ to the nearest integer. Let $n \in \mathbb{N}$ and let $k_n \coloneqq \lceil \log_2(n) \rceil$. Then,
	\begin{enumerate}
	\item For $n \in (2^{k_n-1}, 3 \cdot 2^{k_n-2}]$, we have
		\begin{align*}
		\frac{s(n \cdot 2^{1-k_n})}{2^{1-k_n}} &= n - 2^{k_n-1}.
		\end{align*}				
	\item For $n \in (3 \cdot 2^{k_n-2}, 2^{k_n}]$, we have
		\begin{align*}
		\frac{s(n \cdot 2^{1-k_n})}{2^{1-k_n}} &= 2^{k_n}-n.
		\end{align*}
	\end{enumerate}
\end{lemma}

\begin{proof}
Let $n \in \mathbb{N}$ and $f(n) \coloneqq n \cdot 2^{1-k_n}$, where $k_n \coloneqq \lceil \log_2(n) \rceil$. We first need to show that $f(n)$ is strictly monotonically increasing for $n \in (2^{k_n-1}, 2^{k_n}]$. Let $n_1, n_2 \in (2^{k_n-1}, 2^{k_n}]$, such that $n_1 < n_2$. Let $k_{n_1} \coloneqq \lceil \log_2(n_1) \rceil$ and $k_{n_2} \coloneqq \lceil \log_2(n_2) \rceil$. As $n_1, n_2 \in (2^{k_n-1}, 2^{k_n}]$, we have $k_{n_1} = k_{n_2} = k_n$. In order to show that $f(n)$ is strictly monotonically increasing, we need to show that $f(n_1) < f(n_2)$:
\begin{align*}
		f(n_1) < f(n_2) &\Leftrightarrow f(n_2) - f(n_1) > 0 \\
		&\Leftrightarrow n_2 \cdot 2^{1-k_{n_2}} - n_1 \cdot 2^{1-k_{n_1}}  > 0 \\
		&\Leftrightarrow n_2 \cdot 2^{1-k_n} - n_1 \cdot 2^{1-k_n} > 0 \\
		&\Leftrightarrow (n_2-n_1) \cdot \underbrace{2^{1-k_n}}_{> 0} > 0 \\
		&\Leftrightarrow n_1 < n_2.
\end{align*}
Thus, $f(n)$ is strictly monotonically increasing for $n \in (2^{k_n-1},2^{k_n}]$.
\begin{enumerate}
\item Let $n \in (2^{k_n-1},3 \cdot 2^{k_n-2}]$. Thus, $\lceil \log_2(n) \rceil = k_n$. In particular, $\lceil \log_2(2^{k_n-1}+1) \rceil = k_n$ and $\lceil \log_2(3 \cdot 2^{k_n-2}) \rceil = k_n$.
As $f(n)$ is strictly monotonically increasing for $n \in (2^{k_n-1},2^{k_n}]$ and
\begin{align*}
	f(2^{k_n-1}+1) &= (2^{k_n-1}+1) \cdot 2^{1- \lceil \log_2(2^{k_n-1}+1) \rceil} = (2^{k_n-1}+1) \cdot 2^{1-k_n} \\ 
	&= 1 + 2^{1-k_n} \geq 1 \text{ and } \\
	f(3 \cdot 2^{k_n-2}) &= (3 \cdot 2^{k_n-2}) \cdot  2^{1 - \lceil \log_2(3 \cdot 2^{k_n-2}) \rceil} =   (3 \cdot 2^{k_n-2}) \cdot 2^{1-k_n} = \frac{3}{2},
\end{align*}
we have that $n \cdot 2^{1-k_n} \in [1, \frac{3}{2}]$.
In particular, $s(n \cdot 2^{1-k_n})$ is the distance from $n \cdot 2^{1-k_n}$ to 1, i.e. $s(n \cdot 2^{1-k_n}) = n \cdot 2^{1-k_n} - 1$.
This implies
\begin{align*}
	\frac{s(n \cdot 2^{1-k_n})}{2^{1-k_n}} &= n - 2^{k_n-1} 
\end{align*}
as claimed.
\item Let $n \in (3 \cdot 2^{k_n-2}, 2^{k_n}]$. Thus,  $\lceil \log_2(n) \rceil = k_n$. In particular, $\lceil \log_2(3 \cdot 2^{k_n-2} + 1) \rceil = k_n$ and $\lceil \log_2(2^{k_n}) \rceil = k_n$. 
As $f(n)$ is strictly monotonically increasing for $n \in (2^{k_n-1},2^{k_n}]$ and
\begin{align*}
f(3 \cdot 2^{k_n-2}+1) &= (3 \cdot 2^{k_n-2}+1) \cdot 2^{1-\lceil \log_2(3 \cdot 2^{k_n-2} + 1) \rceil} =(3 \cdot 2^{k_n-2}+1) \cdot 2^{1-k_n} \\
 &= \frac{3}{2} + 2^{1-k_n} \geq \frac{3}{2} \text{ and } \\
	f(2^{k_n}) &= 2^{k_n} \cdot 2^{1-\lceil \log_2(2^{k_n}) \rceil}  = 2^{k_n} \cdot 2^{1-k_n} = 2,
\end{align*}
we have that $n \cdot 2^{1-k_n} \in [\frac{3}{2},2]$. In particular, $s(n \cdot 2^{1-k_n})$ is the distance from $2$ to $n \cdot 2^{1-k_n}$, i.e. $s(n \cdot 2^{1-k_n}) = 2 - n \cdot 2^{1-k_n}$. This implies
\begin{align*}
\frac{s(n \cdot 2^{1-k_n})}{2^{1-k_n}} &= 2^{k_n}-n,
\end{align*}
as claimed. 
\end{enumerate}
This completes the proof.
\end{proof}

\setcounter{theorem}{4}
\begin{theorem} 
Let $T_n^{gfb}$ be a GFB tree with $n$ leaves. Then, $T_n^{gfb}$ has minimal Colless index, i.e. $\mathcal{C}(T_n^{gfb}) = c_n$.
\end{theorem}

\begin{proof}
Let $T_n^{gfb}$ be a GFB tree with $n$ leaves. Let $k_n \coloneqq \lceil \log_2(n) \rceil $.
In order to show that $T_n^{gfb}$ has minimal Colless index, we show that $\mathcal{C}(T_n^{gfb}) = c_n = \sum\limits_{i=0}^{k_n-2} \frac{s(2^{i-k_n+1} \cdot n)}{2^{i-k_n+1}}$. We do this by induction on $n$. 

For $n=1$ we have $\mathcal{C}( T_1^{gfb})=0=c_1$, which gives the base case of the induction.

Now, we assume that the statement given in Theorem \ref{GFB_is_minimal} holds for all GFB trees with up to $n-1$ leaves and we show that it also holds for the GFB tree with $n$ leaves. Now suppose that $n \in (3 \cdot 2^{k_n-2}, 2^{k_n}]$ (the other cases are given in the main part of this manuscript).

By Proposition \ref{GFB_Decomposition} we know that $T_n^{gfb}$ has the following standard decomposition $T_n^{gfb}=(T_{k_n-1}^{fb},T_{n-2^{k_n-1}})$. In particular, $T_{k_n-1}^{fb}$ is the fully balanced tree of height $k_n-1$, and thus by Theorem \ref{min_colless}, $\mathcal{C}(T_{k_n-1}^{fb})=0$. Moreover, for $T_{n-2^{k_n-1}}$ we have: $\lceil \log_2(n-2^{k_n-1}) \rceil = k_n-1$. Thus, using Lemma \ref{colless_sum} and Lemma \ref{max_subtrees}, we have
	\begin{align*}
		\mathcal{C}(T_n^{gfb}) &= \underbrace{\mathcal{C}(T_{k_n-1}^{fb})}_{=0}+ \mathcal{C}(T_{n-2^{k_n-1}}) +  2^{k_n-1} - (n - 2^{k_n-1})  \\
		&= \mathcal{C}(T_{n-2^{k_n-1}}^{gfb}) + 2^{k_n}-n \text{ by Lemma \ref{gfb_subtrees}} \\
		&= c_{n-2^{k_n-1}} + 2^{k_n}-n \; \text{ by the inductive hypothesis} \\
		&= \sum\limits_{i=0}^{(k_n-1)-2} \frac{s(2^{i-(k_n-1)+1} \cdot (n - 2^{k_n-1}))}{2^{i-(k_n-1)+1}} + 2^{k_n}-n \; \text{ by Theorem \ref{colless_explicit}}\\
		&= \sum\limits_{i=0}^{k_n-3} \frac{s(2^{i-k_n+2} \cdot (n - 2^{k_n-1}))}{2^{i-k_n+2}} + 2^{k_n}-n \\
		&= \sum\limits_{i=0}^{k_n-3} \frac{s(2^{i-k_n+2} \cdot n - 2^{i+1})}{2^{i-k_n+2}}  + 2^{k_n}-n \\ 
		&= \sum\limits_{i=0}^{k_n-3} \frac{s(2^{i-k_n+2} \cdot n )}{2^{i-k_n+2}}  + 2^{k_n}-n \; \text{ by Lemma \ref{subbaditivity_of_s}, Part 2, as $2^{i+1} \in \mathbb{Z}$} \\ 
		&= \frac{s(2^{2-k_n} \cdot n)}{2^{2-k_n}} + \frac{s(2^{3-k_n} \cdot n)}{2^{3-k_n}} + \ldots + \frac{s(2^{-1} \cdot n)}{2^{-1}} + 2^{k_n}-n \\
		&=  \frac{s(2^{1-k_n} \cdot n)}{2^{1-k_n}} + \frac{s(2^{2-k_n} \cdot n)}{2^{2-k_n}} + \frac{s(2^{3-k_n} \cdot n)}{2^{3-k_n}} + \ldots + \frac{s(2^{-1} \cdot n)}{2^{-1}}  \\
		&\; \text{as $2^{k_n}-n = \frac{s(2^{1-k_n} \cdot n)}{2^{1-k_n}} $ by Lemma \ref{gfb_technical_lemma}, Part 2}\\
		&= \sum\limits_{i=0}^{k_n-2} \frac{s(2^{i-k_n+1} \cdot n)}{2^{i-k_n+1}} = c_n \; \text{ by Theorem \ref{colless_explicit}.}
		\end{align*}
	Thus, $T_n^{gfb}$ has minimal Colless index. This completes the proof.
\end{proof}

\setcounter{theorem}{5}
\begin{theorem}
Let $T^{gfb}_n = (T_a, T_b)$ be a GFB tree on $n$ leaves with $n \in (2^{k_n-1}, 2^{k_n})$ (where $k_n = \lceil \log_2(n) \rceil$), i.e.
$$ (n_a, n_b) = 
	\begin{cases}
	(n - 2^{k_n-2}, 2^{k_n-2}), &\text{ if } n \in (2^{k_n-1}, 3 \cdot 2^{k_n-2}); \\
	(2^{k_n-1}, 2^{k_n-2}), &\text{ if } n = 3 \cdot 2^{k_n-2}; \\
	(2^{k_n-1}, n-2^{k_n-1}), &\text{ if } n \in (3 \cdot 2^{k_n-2}, 2^{k_n}),
	\end{cases}	
	$$
where we have $n_a-n_b = n-2^{k_n-1}$ in the first case, $n_a-n_b=2^{k_n-2}$ in the second case and $n_a-n_b=2^{k_n}-n$ in the last case. Now, suppose that $\widehat{T} = (\widehat{T}_a, \widehat{T}_b)$ is a tree with a more extreme leaf partitioning, i.e. $\widehat{n}_a - \widehat{n}_b > n_a-n_b$ or, to be more precise,
$$ (\widehat{n}_a, \widehat{n}_b) = 
	\begin{cases}
	(n-2^{k_n-2}+j, 2^{k_n-2}-j) \text{ with } j \in \{1, \ldots, 2^{k_n-2}-1\}, &\text{ if } n \in (2^{k_n-1}, 3 \cdot 2^{k_n-2}); \\
	(2^{k_n-1}+j, 2^{k_n-2}-j) \text{ with } j \in \{1, \ldots, 2^{k_n-2}-1 \}, &\text{ if } n = 3 \cdot 2^{k_n-2}; \\
	(2^{k_n-1} + j, n - 2^{k_n-1} - j) \text{ with } j \in \{1, \ldots, n-2^{k_n-1}-1\}, &\text{ if }  n \in (3 \cdot 2^{k_n-2}, 2^{k_n}).
	\end{cases}$$
Then, we have
$$ c_n = \mathcal{C}(T^{gfb}_n) < \mathcal{C}(\widehat{T}),$$
i.e. $\widehat{T}$ does \emph{not} have minimal Colless index (where the first equality follows from Theorem \ref{GFB_is_minimal}). 
\end{theorem}
\begin{proof}[Proof of Theorem \ref{gfb_extreme}: Parts 2 and 3 (Part 1 is proven in Section \ref{sec_characterization} of the present manuscript)] \leavevmode
\begin{enumerate}
\item[2.] $n = 3 \cdot 2^{k_n-2}$: \\
Let $T^{gfb}_n=(T_a,T_b)$ be the GFB tree on $n$ leaves and let $n_a$ and $n_b$ denote the number of leaves of $T_a$ and $T_b$, respectively.
From Proposition \ref{GFB_Decomposition}, we have that $n_a = 2^{k_n-1}$ and $n_b=2^{k_n-2}$. In particular, $T_a$ and $T_b$ are fully balanced trees of height $k_n-1$ and $k_n-2$, respectively.  As $T^{gfb}_n$ is a GFB tree, by Theorem \ref{GFB_is_minimal} it has minimal Colless index and so do its subtrees $T_a$ and $T_b$ (due to Lemma \ref{max_subtrees}). 
Thus,
\begin{align*}
\mathcal{C}(T^{gfb}_n) &= n_a - n_b + c_{n_a} + c_{n_b} \, \text{ by Lemmas } \ref{colless_sum} \text{ and } \ref{max_subtrees} \text{ and Theorem } \ref{GFB_is_minimal} \\
&= 2^{k_n-1} - 2^{k_n-2} + c_{2^{k_n-1}} + c_{2^{k_n-2}} \\
&= 2^{k_n-1} - 2^{k_n-2} + 0 + 0 \, \text{ by Theorem \ref{min_colless}} \\
&= 2^{k_n-2}.
\end{align*}
Now, suppose that $\widehat{T}=(\widehat{T}_a, \widehat{T}_b)$ with $\widehat{n}_a = n_a+j$ and $\widehat{n}_b = n_b-j$ (and $j \in \{1, \ldots, 2^{k_n-2}-1\}$) is also a tree with minimal Colless index, i.e. $\mathcal{C}(\widehat{T}) = \mathcal{C}(T)=c_n$. 
Supposing that $\widehat{T}$ is a tree with minimal Colless index, we have (again by Lemmas \ref{colless_sum} and \ref{max_subtrees})
\begin{align*}
\mathcal{C}(\widehat{T}) &= \widehat{n}_a - \widehat{n}_b + c_{\widehat{n}_a} + c_{\widehat{n}_b}\\
&= 2^{k_n-1}+j - 2^{k_n-2} + j + c_{2^{k_n-1}+j} + c_{2^{k_n-2}-j} \\
&= 2^{k_n-2} + \underbrace{2j}_{> 0} + \underbrace{c_{2^{k_n-1}+j}}_{\geq 0} + \underbrace{c_{2^{k_n-2}-j}}_{\geq 0} \\
&> 2^{k_n-2}.
\end{align*}
However, this implies $\mathcal{C}(\widehat{T}) > \mathcal{C}(T^{gfb}_n) = 2^{k_n-2}$, i.e. $\widehat{T}$ is not a tree with minimal Colless index, which completes the proof for this subcase.
\item[3.] $n \in (3 \cdot 2^{k_n-2}, 2^{k_n})$: \\
From Proposition \ref{GFB_Decomposition}, we have for $T^{gfb}_n=(T_a,T_b)$ that $n_a=2^{k_n-1}$ and $n_b=n-2^{k_n-1}$. In particular, $T_a$ is the fully balanced tree of height $k_n-1$ and we have $\mathcal{C}(T_a)=c_{n_a}=0$ (by Theorem \ref{min_colless}). For $T_b$ we have $k_{n_b} = \lceil \log_2(n_b) \rceil = k_n-1$.  
Now, consider $\mathcal{C}(T^{gfb}_n)$. By Lemmas \ref{colless_sum} and \ref{max_subtrees} and by Theorem \ref{GFB_is_minimal} we have
\begin{align}
\mathcal{C}(T^{gfb}_n) &= n_a - n_b + c_{n_a} + c_{n_b} \nonumber \\
&= n_a - n_b + c_{n_b}  \, \text{ by Theorem \ref{min_colless}} \label{Colless_T_cnb} \\
&= n_a - n_b + c_{n-2^{k_n-1}} \nonumber \\
&= n_a - n_b + \sum\limits_{i=0}^{k_n-3} \frac{s(2^{i-k_n+2} \cdot (n-2^{k_n-1}))}{2^{i-k_n+2}} \nonumber \\
&= n_a - n_b + \sum\limits_{i=0}^{k_n-3} \frac{s(2^{i-k_n+2} \cdot n - 2^{i+1})}{2^{i-k_n+2}} \nonumber \\
&= n_a - n_b + \sum\limits_{i=0}^{k_n-3} \frac{s(2^{i-k_n+2} \cdot n)}{2^{i-k_n+2}} \, \text{ by Lemma \ref{subbaditivity_of_s}, Part 2, as } 2^{i+1} \in \mathbb{Z} \nonumber \\
&= n_a - n_b + \sum\limits_{i=0}^{k_n-3} \frac{s(2^{i-k_n+2} \cdot (n-j) + 2^{i-k_n+2} \cdot j)}{2^{i-k_n+2}} \label{Colless_T_secondhalf}
\end{align}
Now, suppose that $\widehat{T} = (\widehat{T}_a, \widehat{T}_b)$ with $\widehat{n}_a = n_a+j$ and $\widehat{n}_b=n_b-j$ (and $j \in \{1, \ldots, n-2^{k_n-1}-1\}$) is also a tree with minimal Colless index, i.e. $\mathcal{C}(\widehat{T}) = \mathcal{C}(T^{gfb}_n)=c_n$. Consider $\mathcal{C}(\widehat{T})$. Again, by Lemmas \ref{colless_sum} and \ref{max_subtrees}, we have
\begin{align}
\mathcal{C}(\widehat{T}) &= \widehat{n}_a - \widehat{n}_b + c_{\widehat{n}_a} + c_{\widehat{n}_b} \nonumber \\
&= (n_a+j) - (n_b-j) + c_{\widehat{n}_a} + c_{\widehat{n}_b} \nonumber \\ 
&= n_a - n_b + 2j + c_{\widehat{n}_a}+ c_{\widehat{n}_b} \nonumber \\
&= n_a - n_b + 2j + \underbrace{c_{2^{k_n-1}+j}}_{\geq 0} + \underbrace{c_{n-2^{k_n-1}-j}}_{\geq 0}. \label{Colless_That_secondhalf}
\end{align}
Now, comparing Equations \eqref{Colless_T_cnb} and \eqref{Colless_That_secondhalf}, it directly follows that we have $\mathcal{C}(T^{gfb}_n) < \mathcal{C}(\widehat{T})$ whenever $c_{n_b} < 2j$, which would contradict the minimality of $\mathcal{C}(\widehat{T})$. Thus, we have that $c_{n_b} \geq 2j$. 
We now claim that $c_{n_b} \geq 2j$ implies $k_{\widehat{n}_b} = \lceil \log_2 (\widehat{n}_b) \rceil \in \{k_n-2, k_n-1\}$. 
	\begin{itemize}
	\item Suppose $k_{\widehat{n}_b} \leq k_n-3$. As $\widehat{n}_b = n_b - j$ and $k_{n_b} = k_n-1$, this implies $j \geq 2^{k_n-3}+1$. In particular, $2j \geq 2^{k_n-2}+2$. \\
	Now, consider $c_{n_b}$: As $k_{n_b} = k_n-1$, we have from Lemma \ref{max_min_Colless} that
	$$ c_{n_b} < 2^{k_n-1}-1 = 2^{k_n-2}.$$
	Summarizing the above, we have
	$$ c_{n_b} < 2^{k_n-2} < 2^{k_n-2} + 2 \leq 2j,$$
	which contradicts the assumption that $c_{n_b} \geq 2j$. Thus, $k_{\widehat{n}_b} \geq k_n-2$.
	\item Moreover, as $\widehat{n}_b < n_b$, we clearly have $k_{\widehat{n}_b} \leq k_{n_b} = k_n-1$. 
	\end{itemize}
Thus, in total $k_{\widehat{n}_b} \in \{k_n-2, k_n-1\}$. \\

Moreover, for $\widehat{T}_a$ we have the following: $k_{\widehat{n}_a} = k_n$.
	\begin{itemize}
	\item First of all, as $\widehat{n}_a = n_a + j = 2^{k_n-1} + j$ (with $j \geq 1$), it immediately follows that $k_{\widehat{n}_a} > k_{n_a} = k_n-1$. 
	\item On the other hand, as $\widehat{n}_a < n$, we also have $k_{\widehat{n}_a} \leq k_n$. 
	\end{itemize}
	Thus, $k_{\widehat{n}_a} = k_n$. \\
	
	We now distinguish between two cases:
	\begin{enumerate}
	\item $k_{\widehat{n}_a} = k_n$ and $k_{\widehat{n}_b} = k_n-1$: \\
	Continuing from Equation \eqref{Colless_That_secondhalf} and using Theorem \ref{colless_explicit}, we have
		\begin{align}
		\mathcal{C}(\widehat{T}) &= n_a - n_b + 2j + c_{2^{k_n-1}+j} + c_{n-2^{k_n-1}-j} \nonumber \\
		&= n_a - n_b + 2j + \sum\limits_{i=0}^{k_n-2} \frac{s(2^{i-k_n+1} \cdot (2^{k_n-1} + j))}{2^{i-k_n+1}} + \sum\limits_{i=0}^{k_n-3} \frac{s(2^{i-k_n+2} \cdot (n-2^{k_n-1}-j))}{2^{i-k_n+2}} \nonumber \\
		&= n_a - n_b + 2j + \sum\limits_{i=0}^{k_n-2} \frac{s( 2^i + 2^{i-k_n+1} \cdot j)}{2^{i-k_n+1}} + \sum\limits_{i=0}^{k_n-3} \frac{s(2^{i-k_n+2} \cdot (n-j) - 2^{i+1})}{2^{i-k_n+2}} \nonumber \\
		&= n_a - n_b + 2j + \sum\limits_{i=0}^{k_n-2} \frac{s(2^{i-k_n+1} \cdot j)}{2^{i-k_n+1}} + \sum\limits_{i=0}^{k_n-3} \frac{s(2^{i-k_n+2} \cdot (n-j))}{2^{i-k_n+2}} \nonumber \\
		&\text{ by Lemma \ref{subbaditivity_of_s}, Part 2 as } 2^{i}, 2^{i+1} \in \mathbb{Z} \nonumber \\
		&= n_a - n_b + 2j + \frac{s(2^{1-k_n} \cdot j)}{2^{1-k_n}} + \sum\limits_{i=1}^{k_n-2} \frac{s(2^{i-k_n+1} \cdot j)}{2^{i-k_n+1}} + \sum\limits_{i=0}^{k_n-3} \frac{s(2^{i-k_n+2} \cdot (n-j))}{2^{i-k_n+2}} \nonumber \\
		&= n_a - n_b + 2j + \frac{s(2^{1-k_n} \cdot j)}{2^{1-k_n}} + \sum\limits_{i=0}^{k_n-3} \frac{s(2^{i-k_n+2} \cdot j)}{2^{i-k_n+2}} + \sum\limits_{i=0}^{k_n-3} \frac{s(2^{i-k_n+2} \cdot (n-j))}{2^{i-k_n+2}}. \label{Colless_That_second1_final}
		\end{align}
		Now, as both $T^{gfb}_n$ and $\widehat{T}$ are trees with minimal Colless index, we have (using Equations \eqref{Colless_T_secondhalf} and \eqref{Colless_That_second1_final})
		\begin{align*}
		0 &= \mathcal{C}(T^{gfb}_n) - \mathcal{C}(\widehat{T}) \\
		&= n_a - n_b + \sum\limits_{i=0}^{k_n-3} \frac{s(2^{i-k_n+2} \cdot (n-j) + 2^{i-k_n+2} \cdot j)}{2^{i-k_n+2}} \\
		&\hspace*{5mm} - n_a + n_b - 2j - \frac{s(2^{1-k_n} \cdot j)}{2^{1-k_n}} - \sum\limits_{i=0}^{k_n-3} \frac{s(2^{i-k_n+2} \cdot j)}{2^{i-k_n+2}} - \sum\limits_{i=0}^{k_n-3} \frac{s(2^{i-k_n+2} \cdot (n-j))}{2^{i-k_n+2}} \\
		&= \sum\limits_{i=0}^{k_n-3} \underbrace{\frac{s(2^{i-k_n+2} \cdot (n-j) + 2^{i-k_n+2} \cdot j) - \big( s (2^{i-k_n+2} \cdot (n-j)) + s(2^{i-k_n+2} \cdot j) \big)}{2^{i-k_n+2}}}_{\leq 0 \text{ by Lemma } \ref{subbaditivity_of_s}, \text{ Part 1}} \\
	&\hspace*{5mm} - \underbrace{2j}_{> 0} - \underbrace{\frac{s(2^{1-k_n} \cdot j)}{2^{1-k_n}}}_{\geq 0} \\
	&< 0.
		\end{align*}
	This is a contradiction and thus, $\mathcal{C}(\widehat{T})$ is not minimal. To be precise, $\mathcal{C}(\widehat{T}) > \mathcal{C}(T^{gfb}_n)$, which completes the proof for this subcase. 

	\item $k_{\widehat{n}_a} = k_n$ and $k_{\widehat{n}_b} = k_n-2$: \\
	In this case, we can conclude the following (which will be useful later on):
	\begin{itemize}
		\item On the one hand, we have
			\begin{align*}
			j &\geq n_b - 2^{k_n-2} \; \text{ (in order to have } k_{\widehat{n}_b} \leq k_n-2) \\
			&= n - 2^{k_n-1} - 2^{k_n-2} \\
			&= n- 3 \cdot 2^{k_n-2}.
			\end{align*}
		\item On the other hand, we have
			\begin{align*}
			j &\leq n_b - 2^{k_n-2} + 2^{k_n-3} - 1 \; \text{ (in order to have } k_{\widehat{n}_b} \geq k_n-2) \\
			&< n_b - 2^{k_n-2} + 2^{k_n-3} \\
			&= n- 2^{k_n-1} - 2^{k_n-2} + 2^{k_n-3} \\
			&= n- 2^{k_n-1} - 2^{k_n-3}.
			\end{align*}
		\end{itemize}
		To summarize,
		\begin{align*}
		&\hspace*{5mm} n - 3 \cdot 2^{k_n-2} \leq j < n - 2^{k_n-1} - 2^{k_n-3} \\
		&\Leftrightarrow -n + 3 \cdot 2^{k_n-2} \geq -j > -n + 2^{k_n-1} + 2^{k_n-3} \\
		&\Leftrightarrow 3 \cdot 2^{k_n-2} \geq n-j > 2^{k_n-1} + 2^{k_n-3} = 2 \cdot 2^{k_n-2} + \frac{1}{2} \cdot 2^{k_n-2} \\
		&\Leftrightarrow 2^{2-k_n} \cdot  3 \cdot 2^{k_n-2} \geq 2^{2-k_n} \cdot (n-j) > 2^{2-k_n} \cdot (2 \cdot 2^{k_n-2} + \frac{1}{2} \cdot 2^{k_n-2}) \\
		&\Leftrightarrow 3 \geq 2^{2-k_n} \cdot (n-j) > \frac{5}{2}.
		\end{align*}
		This, however, implies
		\begin{align*}
		\frac{s(2^{2-k_n} \cdot (n-j))}{2^{2-k_n}} &= \frac{3 - 2^{2-k_n} \cdot (n-j)}{2^{2-k_n}} \\
		&= 3 \cdot 2^{k_n-2} - (n-j) \\
		&= 3 \cdot 2^{k_n-2} - n + j.
		\end{align*}
		Now, consider
		\begin{align}
		2j - \frac{s(2^{2-k_n} \cdot (n-j))}{2^{2-k_n}} &= 2j - 3 \cdot 2^{k_n-2} + n - j \nonumber \\
		&= n + j - 3 \cdot 2^{k_n-2} \nonumber \\
		&= \underbrace{j}_{> 0} + \underbrace{n - 3 \cdot 2^{k_n-2}}_{> 0 \text{ as } n \in (3 \cdot 2^{k_n-2}, 2^{k_n})} \nonumber \\
		&> 0. \label{positive}
		\end{align}
	
	We now consider $\mathcal{C}(\widehat{T})$.
	Continuing from Equation \eqref{Colless_That_secondhalf} and using Theorem \ref{colless_explicit}, we have
	\begin{align}
	\mathcal{C}(\widehat{T}) &= n_a - n_b + 2j + c_{2^{k_n-1}+j} + c_{n-2^{k_n-1}-j} \nonumber \\
	&= n_a - n_b + 2j + \sum\limits_{i=0}^{k_n-2} \frac{s(2^{i-k_n+1} \cdot (2^{k_n-1} + j))}{2^{i-k_n+1}} + \sum\limits_{i=0}^{k_n-4} \frac{s(2^{i-k_n+3} \cdot (n-2^{k_n-1} - j))}{2^{i-k_n+3}} \nonumber \\
	&= n_a - n_b + 2j + \sum\limits_{i=0}^{k_n-2} \frac{s( 2^i + 2^{i-k_n+1} \cdot j)}{2^{i-k_n+1}} + \sum\limits_{i=0}^{k_n-4} \frac{s( 2^{i+2} + 2^{i-k_n+3} \cdot (n-j))}{2^{i-k_n+3}} \nonumber \\
	&= n_a - n_b + 2j + \sum\limits_{i=0}^{k_n-2} \frac{s(2^{i-k_n+1} \cdot j)}{2^{i-k_n+1}} + \sum\limits_{i=0}^{k_n-4} \frac{s(2^{i-k_n+3} \cdot (n-j))}{2^{i-k_n+3}} \nonumber \\
	&\text{ by Lemma \ref{subbaditivity_of_s}, Part 2, as } 2^{i} \text{ and } 2^{i+2} \in \mathbb{Z} \nonumber \\
	&= n_a - n_b + 2j +  \frac{s(2^{1-k_n} \cdot j)}{2^{1-k_n}} + \sum\limits_{i=1}^{k_n-2} \frac{s(2^{i-k_n+1} \cdot j)}{2^{i-k_n+1}} + \sum\limits_{i=0}^{k_n-4} \frac{s(2^{i-k_n+3} \cdot (n-j))}{2^{i-k_n+3}} \nonumber \\
	&= n_a - n_b + 2j +  \frac{s(2^{1-k_n} \cdot j)}{2^{1-k_n}} + \sum\limits_{i=0}^{k_n-3} \frac{s(2^{i-k_n+2} \cdot j)}{2^{i-k_n+2}} + \sum\limits_{i=0}^{k_n-4} \frac{s(2^{i-k_n+3} \cdot (n-j))}{2^{i-k_n+3}} \nonumber \\
	&= n_a - n_b + 2j +  \frac{s(2^{1-k_n} \cdot j)}{2^{1-k_n}} + \sum\limits_{i=0}^{k_n-3} \frac{s(2^{i-k_n+2} \cdot j)}{2^{i-k_n+2}} + \sum\limits_{i=1}^{k_n-3} \frac{s(2^{i-k_n+2} \cdot (n-j))}{2^{i-k_n+2}} \nonumber \\
	&= n_a - n_b + 2j +  \frac{s(2^{1-k_n} \cdot j)}{2^{1-k_n}} + \sum\limits_{i=0}^{k_n-3} \frac{s(2^{i-k_n+2} \cdot j)}{2^{i-k_n+2}} + \sum\limits_{i=0}^{k_n-3} \frac{s(2^{i-k_n+2} \cdot (n-j))}{2^{i-k_n+2}} \nonumber \\
	&\hspace*{5mm} - \frac{s(2^{2-k_n} \cdot (n-j))}{2^{2-k_n}} \nonumber \\
	&= n_a - n_b + \sum\limits_{i=0}^{k_n-3} \frac{s(2^{i-k_n+2} \cdot j) + s(2^{i-k_n+2} \cdot (n-j))}{2^{i-k_n+2}} + 2j + \frac{s(2^{1-k_n} \cdot j)}{2^{1-k_n}} \nonumber \\
	&\hspace*{5mm} - \frac{s(2^{2-k_n} \cdot (n-j))}{2^{2-k_n}}. \label{Colless_That_second2_final}
	\end{align}

	Now, as both $\mathcal{C}(T^{gfb}_n)$ and $\mathcal{C}(\widehat{T})$ are minimal, we have (using Equations \eqref{Colless_T_secondhalf} and \eqref{Colless_That_second2_final})
	\begin{align*}
	0 &= \mathcal{C}(T^{gfb}_n) - \mathcal{C}(\widehat{T}) \\
	&= n_a - n_b + \sum\limits_{i=0}^{k_n-3} \frac{s(2^{i-k_n+2} \cdot (n-j) + 2^{i-k_n+2} \cdot j)}{2^{i-k_n+2}} - n_a + n_b - \frac{s(2^{1-k_n} \cdot j)}{2^{1-k_n}} \\
	&\hspace*{5mm}  - \sum\limits_{i=0}^{k_n-3} \frac{s(2^{i-k_n+2} \cdot (n-j)) + s(2^{i-k_n+2} \cdot j)}{2^{i-k_n+2}} - \Big( 2j - \frac{s(2^{2-k_n} \cdot(n-j)}{2^{2-k_n}} \Big) \\
	&= \sum\limits_{i=0}^{k_n-3} \underbrace{\frac{s(2^{i-k_n+2} \cdot (n-j) + 2^{i-k_n+2} \cdot j) - \big( s (2^{i-k_n+2} \cdot (n-j)) + s(2^{i-k_n+2} \cdot j) \big)}{2^{i-k_n+2}}}_{\leq 0 \text{ by Lemma } \ref{subbaditivity_of_s}} \\
	&\hspace*{5mm} - \underbrace{\frac{s(2^{1-k_n} \cdot j)}{2^{1-k_n}}}_{\geq 0} - \underbrace{\Big( 2j - \frac{s(2^{2-k_n} \cdot (n-j))}{2^{2-k_n}} \Big)}_{> 0 \text{ by Eq. } \eqref{positive}} \\
	&< 0.
	\end{align*}
	Again, this is a contradiction. Thus, $\mathcal{C}(\widehat{T})$ cannot be minimal.
	\end{enumerate}
\end{enumerate}
This completes the proof.
\end{proof}

\end{document}